%% file: patchworking_and_mirror_symmetry_u.tex
\documentclass[12pt]{amsart}

 \usepackage{amscd, graphicx}
 \usepackage{tikz-cd}
 \usepackage{ifthen}    
 \usepackage{amssymb}   
 \usepackage{amstext}
 \usepackage{amsfonts}
 \usepackage{amsmath}   
 \usepackage{latexsym}  
 \usepackage{color}
 \usepackage{epsfig}
 \usepackage{pstricks, pst-node,pst-text,pst-tree}
 \usepackage[all]{xy}
 \usepackage{array}
 \usepackage{amsthm}
 \usepackage{eucal, mathrsfs} 
 \usepackage{hyperref}
 \usepackage{mathtools}

\newcommand{\numberset}[1]{\ensuremath{\mathbb{#1}}}    
\newcommand{\C}{\numberset{C}}  
\newcommand{\Q}{\numberset{Q}} 
\newcommand{\R}{\numberset{R}}  
\newcommand{\Z}{\numberset{Z}}  
\newcommand{\PP}{\numberset{P}}  

\newcommand{\T}{\numberset{T}}

\newcommand{\FF}{\numberset{F}}

\newcommand{\F}{\mathcal F}

\newcommand{\inn}[2]{ \langle {#1}, {#2} \rangle}

\theoremstyle{definition}

\newtheorem{thm}{Theorem}[section]
\newtheorem{prop}[thm]{Proposition}
\newtheorem{lem}[thm]{Lemma}
\newtheorem{cor}[thm]{Corollary}
\newtheorem{rem}[thm]{Remark}
\newtheorem{ex}[thm]{Example}
\newtheorem{defi}[thm]{Definition}

\DeclareMathOperator{\supp}{supp}
\DeclareMathOperator{\conv}{Conv} 
 
\DeclareMathOperator{\cone}{Cone} 
\DeclareMathOperator{\codim}{codim}

\DeclareMathOperator{\Hom}{Hom}

\DeclareMathOperator{\pic}{Pic}
\DeclareMathOperator{\dv}{Div}

\DeclareMathOperator{\aff}{Aff}

\DeclareMathOperator{\inter}{Int}

\newboolean{mynotes}\setboolean{mynotes}{true}

\newcommand{\mycomments}[1]{
           \ifthenelse{\boolean{mynotes}}
                      {#1}{}
           }

\begin{document}

\title[{Mirror symmetry and patchworking}]{{Mirror symmetry for tropical hypersurfaces and patchworking}}
\author{Diego Matessi, Arthur Renaudineau}

\begin{abstract}
In the first part of the paper, we prove a mirror symmetry isomorphism between integral tropical homology groups of a pair of mirror tropical Calabi-Yau hypersurfaces. We then apply this isomorphism to prove that a primitive patchworking of a central triangulation of a reflexive polytope gives  a connected real Calabi-Yau hypersurface if and only if the corresponding divisor class on the mirror is not zero.

\end{abstract}

\maketitle

\tableofcontents

\section{Introduction}

\subsection{Tropical mirror symmetry} Mirror symmetry is a duality between a pair of $n$-dimensional Calabi-Yau varieties $X$ and $X^{\circ}$, whereby invariants of the complex geometry of $X^{\circ}$ match with invariants of the symplectic geometry of $X$. There is, by now, a vast literature on mirror symmetry and the phenomenon can be studied at various levels of depth. A first indication of mirror symmetry is that $X$ and $X^{\circ}$ should satisfy the following mirror equality of Hodge numbers
\begin{equation}  \label{mirror:hdg} 
     h^{p,q}(X) = h^{n-p,q}(X^{\circ}).
\end{equation}
Notice that when $p=q=1$, this equality expresses the fact that the complex deformation space of $X^{\circ}$ must have the same dimension as the K\"ahler cone of $X$.  A beautiful construction of mirror pairs of manifolds was discovered by Batyrev \cite{batyr:dual_pol} by considering anti-canonical hypersurfaces in toric varieties associated to dual reflexive polytopes. The construction was later generalized by Borisov \cite{borisov:comp_int} to the case of complete intersections. 
We point out that in general the theory is complicated by the fact that in many cases the Calabi-Yau varieties constructed by this method are singular, therefore the Hodge numbers in \eqref{mirror:hdg} must be replaced by the so called stringy Hodge numbers. The equality is proved for the hypersurface and complete intersection cases in \cite{BB}. 

 In this paper we consider the case of hypersurfaces, assuming that there exists a smooth crepant resolution of both $X$ and $X^{\circ}$. In this context, our first result is a proof of  \eqref{mirror:hdg} using tropical geometry. More precisely, if $\Delta$ and $\Delta^{\circ}$ are two dual $(n+1)$-dimensional reflexive polytopes, we assume that there exists unimodular central subdivisions of both polytopes. These provide resolutions of the ambient toric varieties together with crepant resolutions of the anticanonical hypersurfaces, thus defining a mirror pair of smooth Calabi-Yau varieties $X$ and $X^{\circ}$. The same data can be used to define the tropical versions of the hypersurfaces $X_{\text{trop}}$ and $X_{\text{trop}}^{\circ}$ (see Section \ref{combmirror}). Itenberg, Katzarkov, Mikhalkin and Zharkov (IKMZ)  \cite{trop_hom} defined the tropical homology groups $H_{q}(X_{\text{trop}}; \mathcal F_p)$ for any non-singular tropical variety, as the cellular homology of a certain non-constant cosheaf $\mathcal F_p$ of $\Z$-modules defined over $X_{\text{trop}}$. 
The main theorem in \cite{trop_hom} implies in our context that if $X_{\text{trop}}$ is projective then
\begin{equation} \label{trop:hodge}
	\dim H_{q}(X_{\text{trop}}; \mathcal F_p \otimes \Q) = \dim H^{p,q}(X).
\end{equation}
Our first result is then 
\begin{thm} \label{main:1} Let $A$ be a commutative ring. Given central unimodular subdivisions of both $\Delta$ and $\Delta^{\circ}$ with associated mirror tropical hypersurfaces $X_{\text{trop}}$ and $X^{\circ}_{\text{trop}}$, we have canonical isomorphisms 
\[ H_{q}(X_{\text{trop}}; \mathcal F_p \otimes A ) \cong H_{q}(X_{\text{trop}}^{\circ}; \mathcal F_{n-p} \otimes A). \]
\end{thm}

This theorem corresponds to Theorem \ref{thm:main1}  and Corollary  \ref{cor:main1} where we state and prove it in a more general combinatorial context. Indeed, using the definition of tropical homology given in \cite{brugalle2022combinatorial}, we do not require the subdivision to be convex, i.e. the resolutions of the toric varieties need not be projective.  In the convex case and with $A = \Q$, the mirror relation of Hodge numbers follows from \eqref{trop:hodge}. It is not clear to us what the result means, in the non convex case, in terms of the Hodge numbers of the corresponding (non-projective) resolutions of the hypersurfaces, see Section \ref{subsection:mirrordefi} for details. We also expect that the groups $H_{q}(X_{\text{trop}}; \mathcal F_p \otimes A)$ reppresent graded pieces of a filtration over the homology (with $A$ coefficients) of the complex hypersurface $X$, but this has been proved, for any non-singular projective tropical variety, only for $A = \Q$ \cite{trop_hom}. In particular, when the integral homology of $X$ or $X^{\circ}$ has torsion the ranks of these groups with $A = \FF_2$ may differ from the Hodge numbers. In our case it follows from \cite{lefschetz:ARS} that the tropical homology groups have no torsion when the anticanonical bundle is ample, i.e. when $\Delta$ is a Delzant polytope. For examples of Calabi-Yau varieties whose integral homology has torsion see  \cite{batyr:integral}. Another aspect of Theorem \ref{main:1} is that the isomorphism between the tropical homology groups is canonical, on the other hand it is not known whether there exists a canonical isomorphism between the corresponding Hodge groups, except in some special cases (see \cite{cox:katz}, Section 4.1.3). 

The proof of Theorem \ref{main:1} is inspired by the work of Gross-Siebert \cite{G-Siebert2003} and Yamamoto \cite{yam:trop_contr}, for more details see Section \ref{rel:wrks}.

\subsection{Patchworking and mirror symmetry} Another implication of mirror symmetry is the expectation that invariants of Lagrangian submanifolds in $X$ are reflected into invariants of coherent sheaves in $X^{\circ}$. Roughly speaking, this is the content of Kontsevich's homological mirror symmetry conjecture, which is abstractly formulated into the idea that the derived Fukaya category of Lagrangian submanifolds in $X$ is equivalent to the derived category of coherent sheaves on $X^{\circ}$. Given a real Calabi-Yau variety $X$, i.e. one which is cut out by equations with real coefficients or more generally endowed with an anti-symplectic involution, the real part (or the fixed point set of the involution) is a Lagrangian submanifold. The articles \cite{CBMS}, \cite{arg_princ:realCY}, \cite{matessi:RealCY}, construct and study three dimensional real Calabi-Yau varieties $X$ where the involution preserves a Lagrangian fibration. Using the Strominger-Yau-Zaslow interpretation of mirror symmetry as a duality between Lagrangian fibrations, a family of real structures on $X$ is parametrized by Lagrangian sections, or equivalently by $\FF_2$-divisor classes $D$ in the mirror $X^{\circ}$ \cite{CBMS}.  If we denote by $\R X_D$ the corresponding real part, $\R X_D$ is connected if and only if $D \neq 0$. Moreover, the $\FF_2$-cohomology of $\R X_D$ is determined by the cohomology of the mirror $X^{\circ}$ and the homomorphism $\beta: H^2(X^{\circ}, \FF_2) \rightarrow H^4(X^{\circ}, \FF_2)$ given by $\beta(E) = E^2 + DE$, \cite{arg_princ:realCY}, \cite{matessi:RealCY}. 
 
 It turns out that the map $\beta$ is a boundary map of a spectral sequence which is similar to the one discovered, in any dimension, by the second author and Shaw \cite{ren_shaw:bound} in the context of real hypersurfaces in toric varieties constructed from primitive patchworking. In the second part of this paper we study real Calabi-Yau hypersurfaces, of any dimension $n$, arising from patchworking associated to a primitive central triangulation $T$ of a reflexive polytope $\Delta$. In this case the sign distribution on the integral points of $\Delta$ uniquely determines a divisor $D$ in the Picard group, with $\FF_2$ coefficients, of the resolved dual toric variety $\C \Sigma_T$ associated to $\Delta^{\circ}$. Indeed the integral points on the boundary of $\Delta$ correspond to rays in the fan of $\C \Sigma_T$, and therefore to toric divisors. The divisor $D$ is formed by considering the sum of toric divisors whose rays have the same sign as the center of $\Delta$ (see Figure \ref{pw_divisors} for a simple example). If we denote by $\R X_D$ the real part of the hypersurface $X_D$ constructed from patchworking, it turns out (see Proposition \ref{div:pw}) that if two divisors $D$ and $D'$ are linearly equivalent (over $\FF_2$), then $\R X_D$ and $\R X_{D'}$ are the same up to an automorphism of the ambient toric variety $\C \Sigma_T$. Indeed the signs corresponding to $D$ and $D'$ correspond to two octants of the same patchworking.
 
 The main result of the second part is the following.   

\begin{thm} \label{main:2} Let $X$ be a real Calabi-Yau hypersurface arising from a central, primitive patchworking of a reflexive polytope $\Delta$ such that the tropical homology groups $H_{n}(X_{\text{trop}}; \mathcal F_k \otimes \FF_2)$ vanish for all $0 < k < n$.  Then $\R X_D$ is connected if and only if the tropical divisor class $D_{| X_{\text{trop}}^{\circ}} \in H_{n-1}(X_{\text{trop}}^{\circ}; \mathcal F_{n-1} \otimes \FF_2)$ is zero. 
\end{thm}

\begin{rem} We expect that the tropical divisor class $D_{| X_{\text{trop}}^{\circ}}$ vanishes if and only if the complex divisor $D_{|X^{\circ}}$ also vanishes in the $\FF_2$-Picard group, but we could not find this result in current literature. We point out that the hypothesis on the vanishing of the tropical homology groups is satisfied, for instance, whenever $\Delta$ is a Delzant polytope, i.e. when it corresponds to a smooth toric variety with ample anti-canonical bundle \cite{lefschetz:ARS}, \cite{brugalle2022combinatorial}. Under this hypothesis, the number of connected components of $\R X_D$ can be at most two.
\end{rem}

The theorem corresponds to Theorem \ref{connectedness}, where the result is stated in a more general combinatorial context where the subdivision is not assumed to be convex. To prove it we consider the first page of the spectral sequence in \cite{ren_shaw:bound} converging to the $\FF_2$-homology of the real part $\R X_D$.

The part of the spectral sequence which determines the number of connected components starts with the map $\delta^{[1]}_{D}: H_n( X_{\text{trop}}; \mathcal F_0 \otimes \FF_2) \rightarrow H_{n-1}( X_{\text{trop}}; \mathcal F_1 \otimes \FF_2)$. Now $H_n( X_{\text{trop}};  \mathcal F_0 \otimes \FF_2)$ is one dimensional and generated by a fundamental class $S$. It turns out that if we apply the mirror symmetry isomorphism of Theorem \ref{main:1} to the class $\delta^{[1]}_{D}(S)$, we find that it is the tropical divisor class $D_{| X_{\text{trop}}^{\circ}} \in H_{n-1}( X_{\text{trop}}^{\circ}; \mathcal F_{n-1} \otimes \FF_2)$. If this class is not zero, then $\R X_D$ is connected. If instead it is zero, then we study the maps $\delta^{[k]}_{D}(S)$ in higher pages of the spectral sequence to show that in this case $\R X_D$ has two connected components. 

\subsection{Related works} \label{rel:wrks}
A proof of Theorem \ref{main:1} may also follow by combining the work of Gross-Siebert \cite{G-Siebert2003} on mirror symmetry via integral affine manifolds with singularities, with the work of Yamamoto \cite{yam:trop_contr}. Indeed Gross and Siebert associate to a toric degeneration of Calabi-Yau manifolds, an integral affine manifold with singularities (IAMS), i.e. a topological manifold $B$ with an atlas of integral affine coordinates away from a closed, codimension $2$ subset of $B$ (the dicriminant locus). The manifold $B$ is the dual intersection complex of the special fibre of the toric degeneration and plays a role similar to the tropical variety $X_{\text{trop}}$. Over $B$ there are natural sheaves of $\Z$-modules, denoted $\iota_{\ast} \Lambda^{\wedge p}$, where $\Lambda$ is a local system induced by the affine coordinates and $\iota$ is the inclusion of the complement of the discriminant locus inside $B$. The cohomology groups $H^q(B, \iota_{\ast} \Lambda^{\wedge p})$ play the same role as the tropical homology groups. Two mirror symmetric toric degenerations give dual pairs of integral affine manifolds with singularities $B$ and $B^{\circ}$ related to each other by discrete Legendre transform (see Section 1.4 of \cite{G-Siebert2003}). If $\Lambda^*$ denotes the dual local system, it is shown that the pairs $(B, \iota_{\ast} \Lambda^{\wedge p})$ and $(B^{\circ}, \iota_{\ast} (\Lambda^*)^{\wedge p})$ are isomorphic (see Proposition 1.50 of \cite{G-Siebert2003}). On the other hand, contraction with a global section of $\iota_{\ast} \Lambda^{\wedge n}$ gives isomorphisms  $(\Lambda^*)^{\wedge p} \cong  (\Lambda)^{\wedge {n-p}}$. Thus we obtain the isomorphism $H^q(B, \iota_{\ast} \Lambda^{\wedge p}) \cong H^q(B^{\circ}, \iota_{\ast} \Lambda^{\wedge n-p})$, which is analogous to our Theorem \ref{main:1}. The last step is to link the pairs $(B, \iota_{\ast} \Lambda^{\wedge p})$ with $(X_{\text{trop}}, \mathcal F_p)$, this is achieved by the work of Yamamoto  \cite{yam:trop_contr}. Indeed to an integral affine manifold with singularities, Yamamoto associates an (abstract) tropical variety $X_{\text{trop}}$ together with a correspondence between the sheaves. 

\medskip

Our proof of Theorem \ref{main:1} is inspired by the above works, but it is self-contained and the isomorphism is more explicit since it is expressed in terms of cellular homology. We define cosheaves $\mathcal M_p$ which are supported on a cellular refinement of the union of the bounded cells of $X_{\text{trop}}$, which is topologically a sphere $S$.  The proof then consists of two steps. First we show that $\mathcal M_p$ and $\mathcal F_p$ are linked together by the short exact sequences \eqref{seq:1} and \eqref{seq:2}. These provide the isomorphisms of homology groups $H^q(X_{\text{trop}}, \mathcal F_p) \cong H^q(S, \mathcal M_p)$. If $S^{\circ}$ is the sphere associated to $X_{\text{trop}}^{\circ}$, the next step is to define a cellular isomorphism between $S$ and $S^\circ$ such that the pairs $(S, \mathcal M_p)$ and $(S^{\circ}, \mathcal M_{n-p})$ are isomorphic. The two steps are the content of Theorem \ref{thm:main1} and together they give the isomorphism in Theorem \ref{main:1}.  Notice that the pairs $(S, \mathcal M_p)$ should be isomorphic to the pairs $(B, \iota_{\ast} \Lambda^{\wedge p})$ from the work of Gross-Siebert. 

\medskip 

Our approach is also inspired by the work \cite{brugalle2022combinatorial} where the authors define tropical homology groups also in the case of non-convex unimodular subdivisions, where there is no associated tropical variety. They also extend to this setting the definition of real phase structures from  \cite{ren_shaw:bound}, together with the spectral sequence computing the cohomology of the real part. 

\subsection{Outlook} To prove Theorem \ref{main:2} we explicitly compute the boundary map $\delta^{[1]}_{D}: H_n( X_{\text{trop}}; \mathcal F_0 \otimes \FF_2) \rightarrow H_{n-1}( X_{\text{trop}}; \mathcal F_1 \otimes \FF_2)$ of the spectral sequence from \cite{ren_shaw:bound} and then apply the mirror symmetry isomorphism of Theorem \ref{main:1}. More generally we have the other boundary maps 
\[ \delta^{[1]}_{D}: H_q( X_{\text{trop}}; \mathcal F_p \otimes \FF_2) \rightarrow H_{q-1}( X_{\text{trop}}; \mathcal F_{p+1} \otimes \FF_2). \]
If we compose with the mirror isomorphisms we get maps on the tropical homology of the mirror 
\[ \delta^{[1]}_{D}: H_q( X^{\circ}_{\text{trop}}; \mathcal F_p \otimes \FF_2) \rightarrow H_{q-1}( X^{\circ}_{\text{trop}}; \mathcal F_{p-1} \otimes \FF_2). \]
We expect that these maps have a geometric description involving the divisor class $D$ and standard operations on (co)-homology (such as intersection products), similar to the maps found in  \cite{arg_princ:realCY}, \cite{matessi:RealCY}. We are confident that the explicit description of the mirror symmetry isomorphism given in proof of Theorem \ref{main:1} is suitably adapted to do this computation.

\subsection*{Acknowledgments} 
The authors wish to thank Erwan Brugallé, Mark Gross and Kris Shaw for useful discussions. Diego Matessi was partially supported by the national research project ``Moduli spaces and special varieties'' (PRIN 2022) and he is a member of the INDAM research group GNSAGA. Arthur Renaudineau acknowledges support from the Labex CEMPI (ANR- 11-LABX-0007-01), and from ANR, project ANR-22-CE40-0014.

\section{Toric, tropical and combinatorial geometry}

In this section we review some background material on toric and tropical geometry. Some references are \cite{fulton:toric}, \cite{cox:little:schenck}, \cite{batyr:dual_pol}, \cite{cox:katz} for toric geometry; \cite{BIMS}, \cite{mik_rau:trop} for tropical geometry and \cite{trop_hom}, \cite{brugalle2022combinatorial} for tropical homology.

\subsection{Polytopes and toric varieties}
Let $M \cong \Z^{n+1}$ be a lattice and let $N = \Hom(M, \Z)$ be its dual lattice. We define $M_{\R} = M \otimes \R$ and $N_{\R} = N \otimes \R$. A lattice polytope $\Delta$ in $M_{\R}$ is the convex hull of a finite number of points of $M$. In this paper we will always assume that $\Delta$ is full dimensional. We will denote by greek letters ($\sigma$, $\tau$, etc.) the faces of $\Delta$. The normal fan $\Sigma_{\Delta}$ of $\Delta$ is the collection of convex rational cones $\{ C_{\sigma}\}_{\sigma \preceq \Delta}$ in $N_{\R}$ given by
\[  C_{\sigma} = \{ v \in N_{\R} \, | \, \inn{v}{w-u} \geq 0, \quad \forall u \in \sigma \ \text{and} \ w \in \Delta \}. \]
We will work with complex, real and tropical toric varieties. Let us briefly recall how they are constructed. Let $\Sigma$ be a rational strongly convex polyhedral fan in $N_{\R}$. For any cone $\sigma$ of $\Sigma$ consider the dual cone
$$
\sigma^{\vee} = \{ u \in M_{\R} \, | \, \inn{u}{v} \geq 0, \quad \forall v \in \sigma\}. 
$$
By Gordon's Lemma, the intersection $S_\sigma=\sigma^{\vee}\cap M$ is a finitely generated semigroup. So for any semigroup $S$, one can consider the set $X_\sigma^S:=\Hom(S_\sigma,S)$ of semigroup homomorphisms. If moreover the semigroup $S$ has a topology, we equipp the set $X_\sigma^S$ with the coarsest topology such that for any $\lambda\in S_\sigma$, the  corresponding evaluation map $X_\sigma^S\rightarrow S$ is continuous. In fact, it is enough to do it for a set of generators of $S_\sigma$ (say of cardinality $N$), and the topology coincides with the induced topology from $S^N$. If $\sigma\subset\tau$, there is a continuous restriction map from $X_\sigma^S$ to $X_\tau^S$. 
\begin{defi}
The toric variety associated to the fan $\Sigma$ over the semigroup $S$ is the direct limit 
$$
S\Sigma := \varinjlim X_\sigma^S.
$$
\end{defi}

Classically, the most common semigroups are either $\R$, $\C$ or $\R_+$ (for multiplication). We will also consider the tropical semigroup $\T:=\R\cup{-\infty}$ with the tropical multiplication being the addition. The neutral element $1_\T$ is equal to $0$, while the zero element $0_{\T}$ is $-\infty$. All these semigroups are equipped with the euclidean topology. The torus $N\otimes S^*$ is open in $S\Sigma$, acts on each $X^S_\sigma$ and this action extends to all of $S\Sigma$. The orbits are in correspondence with the cones of $\Sigma$ via the map $\sigma\rightarrow \mathcal{O}^S_\sigma$, where $\mathcal{O}^S_\sigma$ is the orbit of the point $x_\sigma\in X^S_\sigma$ defined by 
 $$x_\sigma(m)=\left\lbrace \begin{array}{c}
      1_S \text{ if } -m\in S_\sigma, \\ 0_S \text{ otherwise }.\end{array}\right.$$ 
For $S=\C$, we obtain the classical construction of complex toric varieties, see Section 1.3 of \cite{fulton:toric}. If $\Sigma=\Sigma_\Delta$ is the normal fan of a lattice polytope, the variety $\C\Sigma$ is projective.  A fan $\Sigma$ is said to be \emph{regular} if every cone can be generated by a subset of a basis of $N$, then the toric variety $\C\Sigma$ is smooth if and only if $\Sigma$ is regular. A polytope in $M_{\R}$ is called \emph{non-singular} if its normal fan is regular. We say $\Sigma$ is \emph{complete} if its support $|\Sigma|$, i.e. the union of its cones, is equal to $N_{\R}$. Completeness is necessary and sufficient for the toric variety $\C \Sigma$ to be compact. For any fan $\Sigma$, there is a natural map from $\C\Sigma$ to $\R_+\Sigma$ induced by the absolute value from $\C$ to $\R_+$ (composition with the logarithm will identify $\R_+\Sigma$ with $\T\Sigma$). In the case where $\Sigma=\Sigma_\Delta$, one has $\R_+\Sigma_{\Delta}=\Delta$ and the map $\C\Sigma_{\Delta}$ to $\Delta$ is the moment map. In general, the real toric variety $\R \Sigma$ can be reconstructed combinatorially directly form the fan $\Sigma$ as follows (see for example \cite{bihan2006}). For any group homomorphism $\xi:M\rightarrow \mathbb{F}_2$, consider $\R_+\Sigma(\xi)$ a copy of $\R_+\Sigma$. Then
\begin{equation}\label{eq:real part}
\mathbb{R} \Sigma\simeq \bigsqcup_{\xi\in \Hom(M,\mathbb{F}_2)}\R_+\Sigma(\xi) / \sim,
\end{equation}
where $(p,\xi)\sim (p',\xi')$ if and only if $p=p'$ and $\xi\vert_{\sigma^{\perp}}=\xi'\vert_{\sigma^{\perp}}$ for the unique cone $\sigma$ such that $p\in \mathcal{O}^{\R_+}_\sigma$.
\subsection{Reflexive polytopes and mirror symmetry}\label{subsection:mirrordefi}
A lattice polytope $\Delta$ is called reflexive if it is defined by a system of inequalities $\inn{v}{u}\leq 1$ with $v\in N$ and $u\in M_{\R}$. In particular, it implies that $\inter(\Delta) \cap M = \{ 0 \}$, i.e.  $0$ is the unique interior lattice point of $\Delta$. 
The dual polytope $ \Delta^{\circ} \subset N_{\R}$ is defined as
\[ \Delta^{\circ} = \{ u \in N_{\R} \, | \, \inn{u}{v} \leq 1, \quad \forall v \in \Delta \} \]
If $\Delta$ is reflexive, then $\Delta^{\circ}$ is a lattice polytope which is also reflexive and $(\Delta^{\circ})^\circ = \Delta$.

A reflexive polytope $\Delta \subset M_{\R}$ defines another fan, in the same space $M_{\R}$ called the face fan of $\Delta$ denoted by $\Xi_{\Delta}$ whose cones are the cones over the faces of $\Delta$.  We have the following relations
\[ \Xi_{\Delta} = \Sigma_{\Delta^{\circ}} \quad \text{and} \quad  \Xi_{\Delta^{\circ}} = \Sigma_{\Delta}.\]

If $\Delta$ is a non-singular reflexive polytope, then the cones of $\Xi_{\Delta^{\circ}}$ are regular. So all faces of $\Delta^{\circ}$ are simplices with no integer points in the interior. Moreover, if $\sigma$ is a facet of $\Delta^{\circ}$, then the convex hull of the origin in $N$ with $\sigma$ is a \emph{unimodular} simplex, meaning that its volume is $\frac{1}{(n+1)!}$, the minimal possible volume among lattice simplices.

Given a pair of dual polytopes $\Delta$ and $\Delta^{\circ}$, we can consider anticanonical subvarieties $X$ and $X^{\circ}$ inside  $\C\Sigma_{\Delta}$ and $\C\Sigma_{\Delta^{\circ}}$ respectively. These are both Calabi-Yau varieties and they constitute a so called \emph{mirror pair}. Typically these varieties, as well as the ambient toric varieties, will be singular. Sometimes singularities can be resolved in the realm of Calabi-Yau varieties by a crepant resolution. The easiest way to do this, if possible, is to consider unimodular convex central subdivisions. 

\begin{defi}
Let $\Delta \subset M_{\R}$ be a reflexive polytope. A triangulation of $\Delta$ with vertices in $M$ is called central if it is obtained by taking convex hulls of the origin in $M$ and any element of a triangulation of the facets of $\Delta$. We say that the subdivision is convex if it is induced by an integral convex piecewise linear function on $\Delta$. We will also denote by $\partial T$ the simplices of $T$ which are contained in the boundary $\partial \Delta$. 
\end{defi}
Let $T$ be a unimodular central triangulation of $\Delta$. Denote by $\Sigma_T$ the regular fan constructed by taking the cones from the origin of $M$ over the simplices in the subdivision of $\Delta$. This fan is a refinement of the fan $\Xi_{\Delta}=\Sigma_{\Delta^{\circ}}$. So it defines a desingularization $ \C\Sigma_{T}$ of $\C\Sigma_{\Delta^{\circ}}$.  

In this article we assume there exist unimodular triangulations $T$ and $T^{\circ}$ of $\Delta$ and $\Delta^{\circ}$ respectively, we do not always assume they are convex. In this case generic anti-canonical hypersurfaces in $\C\Sigma_{T^{\circ}}$ or $\C\Sigma_{T}$ are smooth Calabi-Yau varieties. From now on we will denote by $X$ and $X^{\circ}$ smooth anticanonical hypersurfaces in $\C\Sigma_{T^{\circ}}$ and $\C\Sigma_{T}$ respectively, instead of the singular ones in $\C\Sigma_{\Delta}$ and $\C\Sigma_{\Delta^{\circ}}$. When the triangulations are not convex these may be non-projective.
 In the projective case, i.e. when $T$ and $T^{\circ}$ are both convex, Batyrev and Borisov  \cite{BB}  proved that the Hodge numbers of $X$ and $X^{\circ}$ satisfy the following mirror identity (see also \cite{cox:katz})
\[ h^{p,q}(X) = h^{n-p,q}(X^{\circ}). \]
This relation is in fact proved in Theorem 4.15 of \cite{BB} in the case where the Hodge numbers are replaced by the so called stringy-Hodge numbers, which only depend on the singular Calabi-Yau. Then, Proposition 1.1 of  \cite{BB} states that the stringy Hodge numbers coincide with the Hodge numbers of the smooth, projective crepant resolutions given by $T$ and $T^{\circ}$.  We expect that the mirror identity of Hodge numbers should also hold in the non-projective case, although we could not find a proof of this in the literature. 

\subsection{Combinatorial mirror pairs}\label{combmirror}
Let us introduce first some important notations we will use all along the paper. Given a simplex $\sigma \in T$ 
\begin{itemize}
\item denote by $C(\sigma) \in \Sigma_{T}$ the cone over $\sigma$;
\item denote by $\sigma^{\perp}$ the subspace of $N$  orthogonal to the integral tangent space of $\sigma$;
\item if $\sigma \neq 0$, let $$\sigma_{\infty} = \sigma \cap \partial \Delta. $$
\end{itemize}

Given a cone $\rho \in \Sigma_{T}$ 
\begin{itemize}
\item denote by $S(\rho)$ the simplex $\rho \cap \Delta$;
\item denote by $\hat{\sigma}$ the simplex $S(C(\sigma))$, i.e. the convex hull of $0$ and $\sigma$.
\end{itemize} 
In particular, when $0 \in \sigma$, then $\hat{\sigma}=\sigma$ and $\sigma_{\infty}$ is the unique facet of $\sigma$ contained in $\partial \Delta$. If $0 \notin \sigma$,  then $\sigma_{\infty} = \sigma$.  

Given two fans $\Sigma$ and $\Sigma'$ such that $\Sigma$ is a refinement of $\Sigma'$,  

\begin{itemize}
\item  for every cone $\rho \in \Sigma$, denote by $\min(\rho)$ the smallest cone of $\Sigma'$ containing $\rho$.
\end{itemize}
If $\Sigma$ is the normal fan of a polytope $\Delta$ 
\begin{itemize}
\item denote by $\rho^{\vee}$ the face of $\Delta$ which is normal to $\rho$, and by $F^\vee$ the cone of $\Sigma$ which is normal to the face $F$ of $\Delta$.
\end{itemize}

Assume for a moment that $T$ is a unimodular convex central triangulation of $\Delta$. If $\mu:\Delta\cap M\rightarrow \Z_{\geq 0}$ is a convex function ensuring the convexity of $T$, one can consider its Legendre transform (defined over $N_\R$):
\begin{equation}\label{tropicalhypersurface}
\mu^*(x)=\max_{y\in\Delta\cap M} ( \langle x,y\rangle -\mu(y)),
\end{equation}
and the set
$$
V_\mu:=\left\lbrace x\in N_\R \mid \exists y_1\neq y_2, \:\: \mu^*(x) = \langle x,y_1\rangle -\mu(y_1)=\langle x,y_2\rangle -\mu(y_2)\right\rbrace.
$$
In other words, the set $V_\mu$ is the set of points where the maximum in (\ref{tropicalhypersurface}) is attainted at least twice. 
It is the so-called tropical hypersurface associated to $\mu$: it is a weighted rational balanced polyhedral complex, inducing a polyhedral subdivision of $N_\R$ which is dual (as a poset) to the triangulation $T$ (see for example \cite{BIMS} for more details). Moreover, the compactification of $V_\mu$ in the tropical toric variety $\T\Sigma_\Delta$ induces a subdivision of $\T\Sigma_\Delta$ with underlying poset the subposet of $\Sigma_\Delta\times T$:
\begin{equation} \label{poset:delta}
\Phi:=\left\lbrace (\rho,\sigma)\in\Sigma_\Delta\times T \mid \sigma\subset \rho^\vee \right\rbrace.
\end{equation}
The order on the poset is the inverse of the inclusion, meaning that $(\rho,\sigma)\leq (\rho',\sigma')$ iff $\rho'\subset\rho$ and $\sigma'\subset\sigma$ (see \cite{brugalle2022combinatorial} Remark 2.3).

\begin{defi} \label{dual:trop}
Let $T$ be an unimodular convex central triangulation of $\Delta$ and $T^\circ$ be an unimodular convex central triangulation of $\Delta^\circ$, induced respectively by convex functions $\mu$ and $\mu^\circ$.
We denote by $X_{\text{trop}}$ the compactification of $V_\mu$ inside $\T\Sigma_{T^\circ}$ and by $X_{\text{trop}}^\circ$ the compactification of $V_{\mu^\circ}$ inside $\T\Sigma_{T}$.
\end{defi} 
The tropical variety $X_{\text{trop}}$ induces a subdivision of $\T\Sigma_{T^\circ}$ with underlying poset the subposet of $T^\circ\times T$ defined by
$$\mathcal{P}(T^\circ,T):=\left\lbrace (\tau,\sigma)\in T^\circ\times T \: \mid \: 0\in\tau \text{ and } \sigma\subset \min(C(\tau))^\vee \right\rbrace . $$
Recall that $C(\tau)$ is the cone (in $\Sigma_{T^\circ}$) over $\tau$. The fan $\Sigma_{T^\circ}$ is a refinement of the fan $\Xi_{\Delta^\circ}$, which is the normal fan of $\Delta$. Then by definition $\min(C(\tau))^\vee$ is a face of $\Delta$. The order on the poset $\mathcal{P}(T^\circ,T)$ is also the inverse of the inclusion.
\begin{ex}
\label{ex:cubicanddual}
Let
 \[ \Delta = \conv((-1,-1), \ (-1,2), \ (2,-1)). \] 
 It is a reflexive polygon and its dual is given by 
 \[ \Delta^\circ = \conv((1,1), (-1,0),(0,-1)).\]
  Denote by $T$ and $T^\circ$ their unique primitive central triangulations. They are both convex. We draw in Figure \ref{refl_pol_subdiv} the subdivision of $\T\Sigma_{T^\circ}$ induced by a tropical curve dual to $T$ and some examples of cells in the underlying poset $\mathcal{P}(T^\circ,T)$.
\begin{figure}[!ht] 
\begin{center}
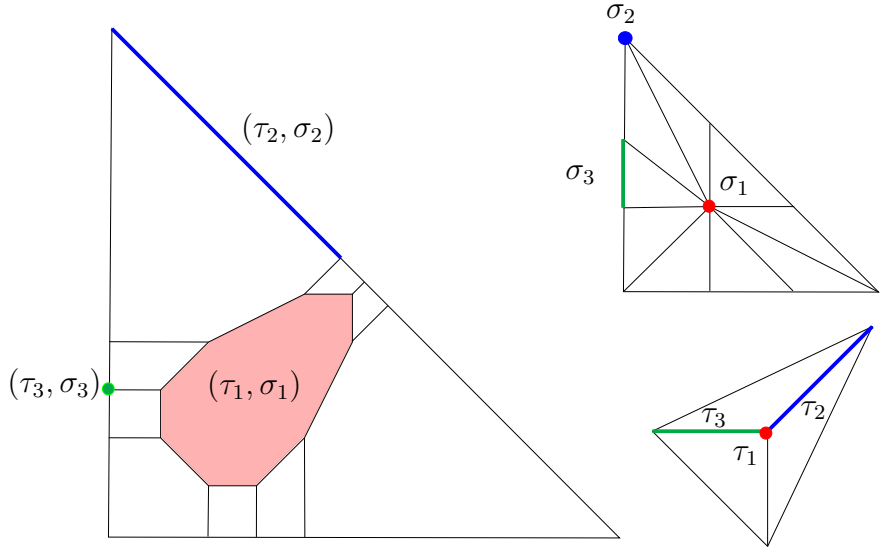
\caption{The subdivision of $\T\Sigma_{T^\circ}$ induced by a tropical curve dual to $T$.}
\label{refl_pol_subdiv}
\end{center}
\end{figure}

\end{ex}
In fact in this paper we do not assume any convexity hypothesis. For any $T$ and $T^\circ$ as above (not necessarily convex), one can still consider the posets $\mathcal{P}(T,T^{\circ})$ and $\mathcal{P}(T^{\circ},T)$. Recall that to any poset $P$ is associated an abstract simplicial complex, called the order complex, whose vertices are the elements of $P$ and whose simplices are the (finite) nonempty chains of $P$. Moreover, to any simplicial complex, there is a topological space called the \emph{geometric realization} defined by formal convex combinations of vertices belonging to a simplex. The geometric realization of a poset $P$ will be denoted by $\vert P\vert $. By definition, the geometric realization is subivided into simplices: for any chain $x_\bullet=(x_1 < \cdots < x_n )$ in  $P$, we denote the corresponding simplex of $\vert P\vert$ by $\vert x_\bullet \vert$. We will consider here a coarser subdivision, defined as follows. For any $x\in P$, consider the set of chains whose elements are smaller than $x$:
$$
x_\leq=\left\lbrace x_\bullet \text{ chains in } P \: \vert  \: x_k\leq x \:\: \forall k \right\rbrace,
$$
and 
$$
F(x)= \bigcup_{x_\bullet \in x_\leq}\vert x_\bullet \vert.
$$
The coarser subdivision of $\vert P\vert$ we consider is defined by 
\begin{equation} \label{poset:subdiv}
\vert P\vert := \bigcup_{x\in P} F(x).
\end{equation}

\begin{ex}
Let $\Sigma$ be a fan. It defines a poset with the order given by reversing the inclusion of cones. Then $\vert \Sigma \vert$ with the coarser subdivision \eqref{poset:subdiv} is isomorphic (as a CW-complex) to the tropical toric variety $\T\Sigma$. The isomorphism is given by $F(\sigma)\rightarrow \overline{\mathcal{O}_\sigma^\T}$ for any cone $\sigma$.
\end{ex}
 The cells $F(\tau,\sigma)$ associated to elements $(\tau, \sigma)$ in the poset $\mathcal{P}(T^\circ,T)$, form a regular CW-structure on $\T\Sigma_{T^\circ}$.
\begin{prop}\label{prop:tropicaltoricposet}
The poset $\mathcal{P}(T^\circ,T)$ is the face poset of a regular CW-structure on $\T\Sigma_{T^\circ}$ with cells $F(\tau,\sigma)$.
\end{prop}
\begin{proof} We can describe the cells $F(\tau, \sigma)$ induced by $\mathcal{P}(T^\circ,T)$ as follows.
Since $\Sigma_{T^\circ}$ is a refinement of $\Sigma_{\Delta}$, we have a blow-down map 
\[ \operatorname{Bl}: \T\Sigma_{T^\circ} \rightarrow \T\Sigma_{\Delta}.\] 
If $\rho \in \Sigma_{T^\circ}$, then $\operatorname{Bl}$ maps the cell $F(\rho)$ to the cell $F(\min \rho)$. In the interior of these cells, the map is given by the quotient
\[ \operatorname{Bl}_{\rho}: N_{\R}/ \langle \rho \rangle \rightarrow  N_{\R}/ \langle \min \rho \rangle. \]
It is shown in \cite{brugalle2022combinatorial} Section 3.2.1 that the cells $F(\rho, \sigma)$ associated to elements in the poset $\Phi$ (see \eqref{poset:delta}) define a regular CW-structure on $\T\Sigma_{\Delta}$. Then, given $(\tau, \sigma) \in \mathcal{P}(T^\circ,T)$,  the cell $F(\tau, \sigma)$ can be identified with the preimage of the cell $F(\min C(\tau), \sigma)$ with respect to $\operatorname{Bl}$, i.e.
\[ F(\tau, \sigma) = \operatorname{Bl}_{C(\tau)}^{-1}(F(\min C(\tau), \sigma)). \]
\end{proof}

Again, in the case of convex triangulations, this subdivision is the subdivision induced on $\T\Sigma_{T^\circ}$ by $X_{\text{trop}}$.
We will be intersted more specifically in the subposet of $\mathcal{P}(T^\circ,T)$ defined by 
$$
\mathcal{P}^1(T^\circ,T):=\left\lbrace (\tau,\sigma)\in \mathcal{P}(T^\circ,T) \vert \dim(\sigma)\geq 1 \right\rbrace .
$$
In the convex case, this poset parametrizes $X_{\text{trop}}$. As it will appear later, this poset will be the support of the cosheaves we will consider. Denote by $X_{T^\circ,T}$ the geometric realization $\vert \mathcal{P}^1(T^\circ,T)\vert$ as defined above with the subdivision \eqref{poset:subdiv}. Define the following subspace of $X_{T^\circ,T}$:
\begin{equation} \label{trop:sphere}
S_{T^\circ,T}=\left\lbrace F(0,\sigma)\in X_{T^\circ,T} \:\: \vert \:\: 0\in\sigma \right\rbrace.
\end{equation}
This is topologically a sphere and in the case where $X_{T^\circ,T}=X_{\text{trop}}$, this is the union of all bounded faces of $X_{\text{trop}}$ in $N_\R$.
\begin{defi}
Let $T$ be a unimodular central triangulation of a reflexive polytope $\Delta$ and $T^\circ$ a unimodular central triangulation of its dual $\Delta^\circ$. The pair of CW-complexes 
$$(X_{T^\circ,T},X_{T,T^\circ}) $$
is called a combinatorial mirror pair. 
\end{defi}

\subsection{Poset homology}
We define the homology groups $H_{\bullet}(P; \mathcal F)$ of a cosheaf $\mathcal F$ defined over a poset $P$. See \cite{brugalle2022combinatorial} and the references therein for more details. View a poset $P$ as a category, whose objects are its elments and whose morphisms are the ordered pairs $x\leq_P y$.  For any ring $R$, an $R$-cosheaf $\mathcal{F}$ on $P$ is a contravariant functor $\iota_\mathcal{F}$ from $P$ to the category of $R$-modules. Given additional assumptions on $P$, one can associate a differential complex $(C_\bullet (P;\mathcal{F}),\partial)$ to any cosheaf $\mathcal{F}$. A \emph{cover relation} is a pair $x\lessdot y$ such that there exists no $z\in P$ with $x< z<y$.  A \emph{grading} on $P$ is a function $\dim: P\rightarrow \Z$ such that $\dim(y)-\dim(x)=1$ for any cover relation $x\lessdot y$ . An interval of length $2$ is an interval $\left[ x,y \right]$ such that $\dim(y)-\dim(x)=2$. We say that $P$ is \emph{thin} if every interval of length $2$ contains exactly $4$ elements. A \emph{signature} is a map $s$ from the set of all cover relations of $P$ to $\{\pm 1\}$, and it is called \emph{balanced} if any interval of length $2$ contains an odd number of $-1$'s. Given a graded, thin poset $P$ with a balanced signature, the differential complex is defined by 
$$
C_q(P;\mathcal{F})=\bigoplus_{\dim(x)=q}\mathcal{F}(x), \:\:\:\: \partial: C_q(P;\mathcal{F})\rightarrow C_{q-1}(P;\mathcal{F}),
$$
where for all $x\in P$ of dimension $q$, one has
$$
\partial\vert_{\mathcal{F}(x)}(a)=\sum_{y\lessdot x}s(y\lessdot x)\iota(y\leq_P x)(a).
$$
Since the poset is thin, we have $\partial^2=0$. The homology groups of  $(C_\bullet (P;\mathcal{F}),\partial)$ are denoted as usual by $H_\bullet (P;\mathcal{F})$. If a subset $U$ of $P$ is closed under taking larger elements (sometimes $U$ is called open), one can restrict the differential complex and the homology groups to $U$. We denote this restriction by $C_\bullet (U;\mathcal{F})$ and $H_\bullet (U;\mathcal{F})$. Also maps of cosheaves induce maps between homology groups, and short exact sequences of cosheaves induce long exact sequences of homology groups.

\subsection{Tropical homology}
Tropical homology is a homology theory with non-constant coefficients defined over tropical varieties \cite{trop_hom}, \cite{brugalle2022combinatorial}. In our case, if $T$ and $T^\circ$ are convex unimodular central triangulations of $\Delta$ and $\Delta^\circ$ respectively, the tropical varieties $X_{\text{trop}}\subset \T\Sigma_{T^\circ}$ and $X_{\text{trop}}^\circ\subset \T\Sigma_{T}$ of Definition \ref{dual:trop} are non-singular projective tropical varieties. The tropical homology groups are denoted by $H_q(X_{\text{trop}}; \mathcal{F}_p)$, and it follows from \cite{trop_hom} that 
$$
\dim H_q(X_{\text{trop}};\mathcal{F}_p\otimes\Q)=h^{p,q}(X) \text{ and } \dim H_q(X_{\text{trop}}^\circ;\mathcal{F}_p\otimes \Q)=h^{p,q}(X^\circ).
$$
Here $X$ and $X^\circ$ are any smooth anticanonical hypersurfaces in $\C \Sigma_{T^\circ}$ and $\C \Sigma_{T}$ respectively. Tropical homology (in it's cellular version) belongs to the realm of poset homology.

We fix $T$ and $T^\circ$ not necessarily convex primitive central triangulations of $\Delta$ and $\Delta^\circ$ respectively. The poset $\mathcal{P}(T^\circ,T)$ is graded by the formula
$$
\dim(\tau,\sigma)=\codim(\tau)-\dim(\sigma).
$$
Note that $\dim(\tau,\sigma)=\dim F(\tau,\sigma)$ in the geometric realization of $\mathcal{P}(T^\circ,T)$.
Moreover, the poset $\mathcal{P}(T^\circ,T)$ is thin and it admits a balanced signature, obtained by choosing orientations on each $F(\tau,\sigma)$.
The $p$-multitangent cosheaves over $\mathcal P(T^\circ,T)$ are defined by
$$
\mathcal{F}_p(\tau,\sigma):=\sum_{
\tiny{\begin{array}{c}
\lambda\subset\sigma \\ \dim(\lambda)=1
\end{array}}} \bigwedge^p (\lambda^\perp / \langle C(\tau) \rangle ) ,
$$
where $\langle C(\tau) \rangle$ is the submodule generated by $C(\tau) \cap N$.
The cosheaf maps of $\mathcal{F}_p$ are induced by the projections 
\begin{equation}\label{cosheafmaps}
N /\langle C(\tau) \rangle \rightarrow N /\langle C(\tau') \rangle,
\end{equation}
if $\tau\subset \tau'$ (recall that by definition of $\mathcal P(T^\circ,T)$, the faces $\tau$ and $\tau'$ contain $0$). 
Note that the support of the cosheaves $\mathcal{F}_p$ is the poset $\mathcal P^1(T^\circ,T)$. In order to lighten the notations, if $\mathcal{G}$ is any poset on $\mathcal P^1(T^\circ,T)$, we will denote by $H_\bullet (X_{T^\circ,T};\mathcal{G})$ the homology groups $H_\bullet (\mathcal P^1(T^\circ,T);\mathcal{G})$.
\section{Combinatorial mirror symmetry}

Define $\mathcal{P}^\infty(T^\circ,T)$ to be the subposet of $\mathcal{P}(T^\circ,T)$ consisting of elements whose first coordinate is non zero. Its elements parametrize the boundary of $\T\Sigma_{T^\circ}$.
\begin{lem}\label{lem:involution_p}
The map 
\begin{equation}
  \begin{split} 
           p_{T^\circ,T}: \mathcal{P}^\infty(T^\circ, T) & \longrightarrow  \mathcal{P}^\infty(T , T^\circ)  \\
             (\tau,\sigma) & \mapsto (\hat{\sigma},\tau_\infty) 
   \end{split} 
 \end{equation}
 is an isomorphism of posets.
\end{lem} 
\begin{proof}
The map $p_{T^\circ,T}$ is clearly invertible with inverse $p_{T,T^\circ}$. We only have to verify that if $\sigma\subset \min(C(\tau))^\vee$, then $\tau_\infty\subset \min(C(\hat{\sigma}))^\vee$. 
Let us say that $\min(C(\tau))=\cone(u_1,\cdots,u_r)$ and that $\min(C(\hat{\sigma}))=\cone(v_1,\cdots,v_s)$, where the $u_i$'s  in $N$ and the $v_j$'s in $M$ are respectively primitive generators of the rays of $\min(C(\tau))$ and  $\min(C(\hat{\sigma}))$. Then since $\sigma\subset \min(C(\tau))^\vee$ and by reflexivity, one has that
$$
 \langle u_i,x \rangle =1,  \quad \forall x\in\sigma \text{ and } \forall i.
$$
Moreover, the same equality is true for any $x$ in the minimal face of $\Delta^{\circ}$ containing $\sigma$. So one has that $\langle u_i,v_j \rangle=1$ for all $i,j$. Now if $x\in\tau_\infty$, then $x=\sum\alpha_i u_i$ with $\sum\alpha_i=1$. Then for all $j$, we obtain that $\langle x,v_j \rangle =\sum \alpha_i  \langle u_i,v_j \rangle=1$. 
\end{proof}
\begin{ex}
In Figure \ref{refl_pol_subdiv_2}, we represented the geometric realization of the poset $\mathcal{P}(T,T^\circ)$, where $T$ and $T^\circ$ are defined as in Example \ref{ex:cubicanddual}, with some examples of cells. Notice that the cells $(\tau_2, \sigma_2)$ and $(\tau_3, \sigma_3)$ appearing in Figure \ref{refl_pol_subdiv_2} correspond, under the map $p_{T^\circ,T}$, to the cells with the same label in Figure \ref{refl_pol_subdiv}.

\begin{figure}[!ht] 
\begin{center}
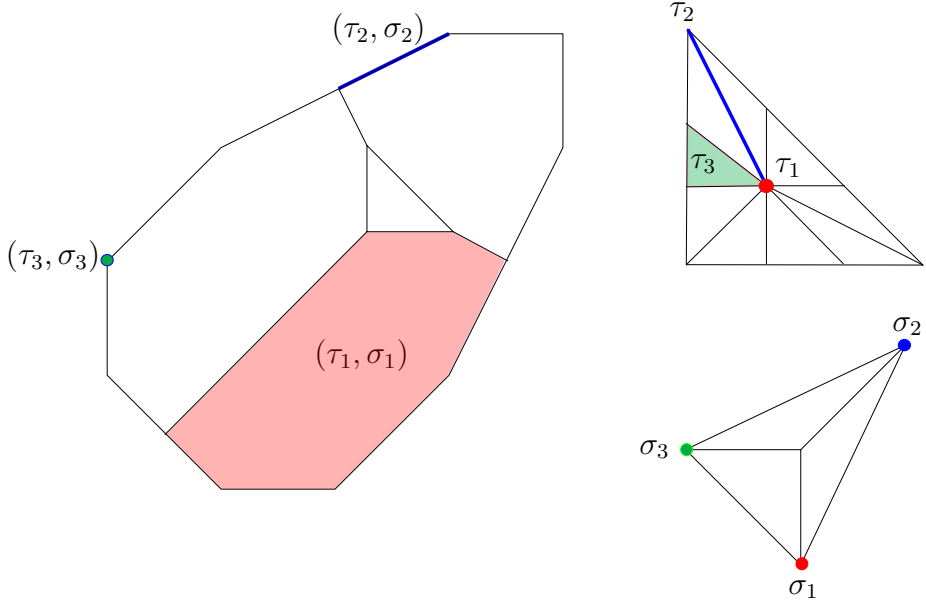
\caption{The subdivision $\mathcal{P}(T,T^\circ)$}
\label{refl_pol_subdiv_2}
\end{center}
\end{figure}
\end{ex}

We will next refine the poset $\mathcal{P}(T^\circ,T)$, in order to extend this correspondence to the whole poset minus the origin.

\subsection{A refinement of the poset $\mathcal{P}(T^\circ,T)$}
\begin{defi}
The poset $\mathcal{J}(T^\circ,T)$ is the subposet of $T^\circ\times T$ defined by

\begin{itemize}
\item If $\sigma\neq 0$ and $0\in\sigma$, then consider all pairs $(\tau,\sigma)$ such that $\tau\in \partial T^\circ$ and $\sigma_\infty\subset \min(C(\tau))^\vee$.
\item If $\sigma\in \partial T$, consider all pairs  $(\tau,\sigma)$ such that $\tau\neq 0$ and $\sigma\subset \min(C(\tau))^\vee$.
\end{itemize}
Define $\mathcal{J}^S(T^\circ,T)$ the subposet of $\mathcal{J}(T^\circ,T)$ consisting of pairs of the first kind, and by  $\mathcal{J}^\infty(T^\circ,T)$ the subposet of $\mathcal{J}(T^\circ,T)$ consisting of pairs $(\tau,\sigma)$ such that $0\in\tau$. 
\end{defi}

Notice that we do not consider the pair $(0,0)$ in $\mathcal{J}(T^\circ,T)$. In fact, by definition the pair $(0,0)$ would be isolated (i.e. not comparable to any other pair inside $\mathcal{J}(T^\circ,T)$).

It follows from the definition that the map from $\mathcal{J}(T^\circ,T)$ to $\mathcal{P}(T^\circ,T)\setminus (0,0)$ sending $(\tau,\sigma)$ to $(0,\sigma)$ if $0\notin \tau$, and to itself if $0\in\tau$ is surjective. The poset $\mathcal{J}(T^\circ,T)$ is still a graded poset (with the same grading as before), thin and still admits a balanced signature. In fact, its geometric realization (as defined above)  is a subdivision of the one of $\mathcal{P}(T^\circ,T)$ (minus the cell $F(0,0)$). Similarly, the geometric realization of $\mathcal{J}^S(T^\circ,T)$ is a subdivision of $S_{T^\circ,T}$, and $\mathcal{J}^\infty(T^\circ,T)=\mathcal{P}^\infty(T^\circ,T)$. As a corollary of Lemma \ref{lem:involution_p}, one obtains the following
\begin{lem}
The map
\begin{equation} \label{mirror_subdiv}  
           j_{T^\circ,T}: \mathcal{J}(T^\circ, T) \longrightarrow  \mathcal{J}(T , T^\circ) 
 \end{equation}
 sending 
 \begin{itemize}
 \item $(\tau,\sigma)$ to $(\hat{\sigma},\tau_\infty)$ if $0\in\tau$ and $\sigma\in \partial T$,
 \item $(\tau,\sigma)$ to $(\sigma_\infty,\hat{\tau})$ if  $\tau\in \partial T^\circ$ and $0\in\sigma$, and
 \item $(\tau,\sigma)$ to $(\sigma,\tau)$ if $(\tau,\sigma)\in \partial T^\circ \times \partial T$ 
 \end{itemize}
 is an isomorphism of posets.
\end{lem}
Notice that $j_{T^\circ,T}$ sends $\mathcal{J}^S(T^\circ,T)$ to $\mathcal{J}^S(T,T^\circ)$, sends $\mathcal{J}^\infty(T^\circ,T)$ to $\mathcal{J}^\infty(T,T^\circ)$ and preserves the dimensions.

\begin{ex}
In Figures \ref{refinfed_subdiv} and \ref{refinfed_subdiv_2} we represented the geometric realizations of the posets $\mathcal{J}(T,T^\circ)$ and $\mathcal{J}(T^\circ, T)$, where $T$ and $T^\circ$ are defined again as in Example \ref{ex:cubicanddual}. The cells depicted in the two figures are dual in the sense that they correspond under the map  $j_{T^\circ,T}$.

\begin{figure}[!ht] 
\begin{center}
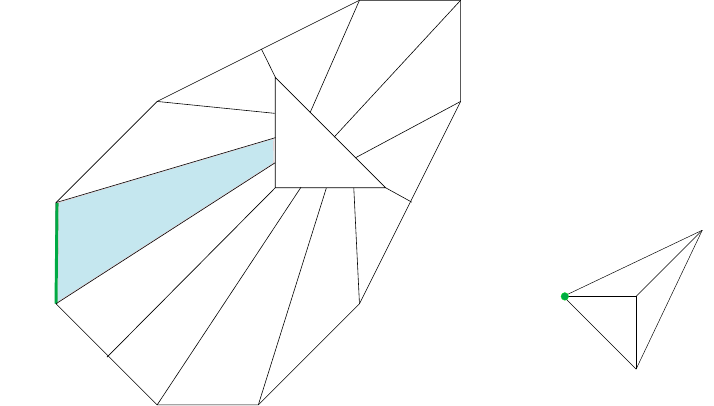
\caption{The subdivision $\mathcal{J}(T,T^\circ)$}
\label{refinfed_subdiv}
\end{center}
\end{figure}
\end{ex}

We extend the cosheaves $\mathcal{F}_p(\tau,\sigma)$ to the poset $\mathcal{J}(T^{\circ}, T)$ by precomposing with the map from above sending $(\tau,\sigma)$ to $(0,\sigma)$ if $0\notin \tau$ and to itself if $0\in\tau$. Define as before
$$
\mathcal{J}^1(T^\circ, T):=\left\lbrace (\tau,\sigma)\in \mathcal{J}(T^\circ, T) \vert \dim(\sigma)\geq 1 \right\rbrace .
$$
The geometric realization of $\mathcal{J}^1(T^\circ, T)$ defines a subdivision of $X_{T^\circ,T}$ and it is the support of the cosheaf $\mathcal{F}_p$. One has a canonical isomorphism $$H_q(\mathcal{J}(T^\circ,T); \mathcal F_p) \cong H_q(X_{T^\circ,T}; \mathcal F_p).
$$

\begin{figure}[!ht] 
\begin{center}
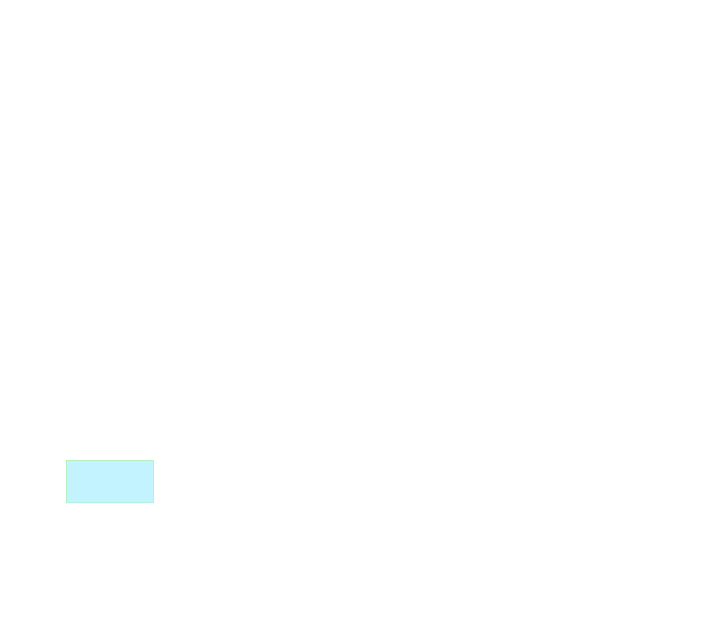
\caption{The subdivision $\mathcal{J}(T^\circ, T)$}
\label{refinfed_subdiv_2}
\end{center}
\end{figure}

\subsection{The mirror cosheaf}
Let $(\tau,\sigma)\in \mathcal{J}^S(T^\circ,T)$. Since $\sigma_\infty\subset \min(C(\tau))^\vee$, in particular the cone $C(\tau)$ is a subset of $\sigma_\infty^\perp$. Then $\langle C(\tau) \rangle \wedge \mathcal{F}_{p-1}(0,\sigma_{\infty})$ is a submodule of $\mathcal{F}_p(0,\sigma)$, and we can consider the quotient.
\begin{defi}
Let $p\geq 0$. The $p$-th mirror  cosheaf $\mathcal{M}_p$ is defined over $\mathcal{J}(T^\circ,T)$ as follows: 
\begin{itemize}
\item if $\sigma\neq 0$ and $0\in\sigma$, then define 
\begin{equation} \label{eqMp}
\mathcal{M}_p(\tau, \sigma)=\frac{\mathcal{F}_p(0,\sigma)} { \langle C(\tau) \rangle \wedge \mathcal{F}_{p-1}(0,\sigma_{\infty})}
\end{equation}
where $\langle C(\tau) \rangle$ is the submodule generated by $C(\tau) \cap N$. When $p=0$, we define $\mathcal{M}_0(\tau, \sigma) = \Z$.
\ \\ 

\item for all other  elements $(\tau,\sigma) \in  \mathcal{J}(T^\circ,T)$ set
\begin{equation} \label{eq0}
\mathcal{M}_p( \tau, \sigma)=0.
\end{equation}
\end{itemize}
The cosheaf maps are induced by the ones from $\mathcal{F}_p$.
\end{defi}
Notice that the support of $\mathcal{M}_p$ is $\mathcal{J}^S (T^\circ,T)$, and that moreover
\[  \mathcal M_p(\tau, \sigma) = \mathcal F_p(0, \sigma) \quad \text{when } 0\in\sigma \text{ and } \dim \sigma = 1.\]

We have the following main result

\begin{thm} \label{thm:main1} Let $A$ be a commutative ring. For all $q \geq 0$, there are canonical isomorphisms
\begin{equation} \label{FtoM} 
H_q(\mathcal{J}(T^\circ,T); \mathcal F_p\otimes A) \cong H_q(\mathcal{J}(T^\circ,T); \mathcal M_p\otimes A),
\end{equation}
\begin{equation} \label{M_ms}
   H_q(\mathcal{J}(T^\circ,T); \mathcal M_p\otimes A) \cong H_q(\mathcal{J}(T,T^\circ); \mathcal M_{n-p}\otimes A).
\end{equation}
Moreover, they induce a canonical isomorphism
\begin{equation} \label{trop_ms}
H_q(X_{T^\circ,T}; \mathcal F_p\otimes A) \cong H_q(X_{T,T^\circ}; \mathcal F_{n-p}\otimes A).
\end{equation}
\end{thm} 
\begin{cor} \label{cor:main1}
If $T$ and $T^\circ$ are convex subdivisions, and $X_{\text{trop}}$ and $X^\circ_{\text{trop}}$ are some correponding tropical hypersurfaces as in Definition  \ref{dual:trop}, then we have the mirror symmetry of the tropical Hodge diamonds
\[ \dim H_q( X_{\text{trop}}; \mathcal F_p\otimes \Q) = \dim H_q(X^{\circ}_{\text{trop}}; \mathcal F_{n-p}\otimes \Q). \]
\end{cor}
Combining this equality with \cite{trop_hom} gives another proof of the mirror identity \eqref{mirror:hdg} proved by Batyrev and Borisov in \cite{BB}.
\begin{defi}[Mirror classes] 
Two classes $\gamma\in H_q(\mathcal{P}(T^\circ,T); \mathcal F_p\otimes A)$ and  $\gamma'\in H_q(\mathcal{P}(T,T^\circ); \mathcal F_{n-p}\otimes A)$ are called \emph{mirror classes} if they coincide under isomorphism (\ref{trop_ms}). 
\end{defi}
\subsection{Proof of Theorem \ref{thm:main1}}
We begin by extending the support of $\mathcal{M}_p$ to the whole poset $\mathcal{J}(T^\circ,T)$.
\begin{defi}
Let $p\geq 0$. The cosheaf $\mathcal{M}_p^{\Delta}$ is defined over $\mathcal{J}(T^\circ,T)$ by
\begin{equation} \label{eq1}
\mathcal{M}_p^{\Delta}(\tau, \sigma)=\dfrac{\mathcal{F}_p(0,\sigma )}{\langle C(\tau) \rangle \wedge \mathcal{F}_{p-1}(0,\sigma_\infty)}
\end{equation}
if $\tau\in \partial T^\circ$, and by $\F_p$ otherwise.
Notice that
\[  \mathcal{M}_p^{\Delta}(\tau, \sigma ) = \mathcal{M}_p(\tau, \sigma ) \quad \text{when} \  (\tau, \sigma) \in \mathcal{J}^{S}(T^\circ,T)  \]
and 
\[  \mathcal{M}_p^{\Delta}(\tau, \sigma ) = \mathcal{F}_p(\tau ,\sigma ) \quad \text{if} \  0\in \tau. \]
\end{defi}

By definition, we have an inclusion $\mathcal{M}_p \hookrightarrow \mathcal{M}_p^{\Delta}$ and a surjection $\mathcal F_p \rightarrow \mathcal M_p^{\Delta}$. These give two short exact sequences of cosheaves
\begin{equation} \label{seq:1}
      0 \rightarrow \mathcal M_p \rightarrow \mathcal M_p^{\Delta} \rightarrow \mathcal Q_p \rightarrow 0,
\end{equation}
where $\mathcal Q_p$ is the quotient cosheaf, and 
\begin{equation} \label{seq:2}
 0 \rightarrow \mathcal R_p \rightarrow  \mathcal F_p \rightarrow \mathcal M_p^{\Delta} \rightarrow 0,  
 \end{equation}
where $\mathcal R_p$ is the kernel.  

The quotient $\mathcal{Q}_p$ is given by $\mathcal{Q}_p(\tau, \sigma)=\mathcal{M}_p^{\Delta}(\tau, \sigma )$
if $\sigma\in \partial T$ and is equal to zero otherwise. In particular $\mathcal Q_p$ is supported on 
\[ \mathcal{J}^{ub}(T^\circ,T):=\mathcal{J}(T^\circ,T)\setminus \mathcal{J}^S(T^\circ,T). \]
The kernel $\mathcal{R}_p$ is given by $\mathcal{R}_p(\tau,\sigma)=\langle C(\tau) \rangle \wedge \mathcal{F}_{p-1}(0,\sigma_\infty)$ if $\tau\in \partial T^\circ$ and is equal to zero otherwise. The support of $\mathcal{R}_p$ is 
\[ \mathcal{J}^{0}(T^\circ,T):=\mathcal{J}(T^\circ,T)\setminus \mathcal{J}^\infty(T^\circ,T). \]
We have
\begin{prop} \label{q_iso}
The cosheaves  $\mathcal M_p$,  $\mathcal M_p^{\Delta}$, $\mathcal{Q}_p$ and $\mathcal R_p$ are cosheaves of free $\Z$-modules. Moreover, let $\sigma\in \partial T$ and $\tau\in \partial T^\circ$ such that $\sigma\subset \min(C(\tau))^\vee$. Then the cosheaf maps 
\[ \begin{split}  \mathcal Q_p(\tau, \sigma) & \rightarrow \mathcal Q_p (\hat{\tau}, \sigma), \\
                        \mathcal R_p(\tau, \sigma) & \rightarrow \mathcal R_p (\tau, \hat{\sigma} )
   \end{split}                \] are isomorphisms. 

\end{prop}
\begin{proof}   
It is easy to see that $\mathcal F_p$ is a cosheaf of free $\Z$-modules and therefore also $\mathcal R_p$.  It is then enough to show that $\mathcal M^{\Delta}_p$ is a cosheaf of free modules. 

Let $(\tau, \sigma) \in \mathcal{J}(T^\circ,T)$ with $0 \in \sigma$ and $\tau \in \partial T^{\circ}$.  We begin by describing bases for $\F_p(0, \sigma)$,   $\F_p(0, \sigma_{\infty})$ and $\langle C(\tau) \rangle \wedge \mathcal{F}_{p-1} (0,\sigma_{\infty})$. The cell $\sigma$ is a unimodular simplex, let $\dim \sigma= s+1$. Choose a vertex $e_0$ of $\sigma_{\infty}$ and fix $s$ edges of $\sigma_{\infty}$ containing this vertex. Denote by $e_1,\cdots,e_{s}$ a choice of primitive vectors directing those edges.  Complete the set $e_1, \ldots, e_{s}$ to a basis of $\langle C(\tau)\rangle^\perp$ denoted by $e_1,\cdots,e_m$, and complete the set $e_0, \ldots, e_m$ to a basis of $M$ denoted by $e_0,\cdots,e_{n}$. Consider the dual basis $e_0^*,\cdots,e_{n}^*$ of $N$. We have 
$$\langle C(\tau) \rangle = \langle e_0^*, e_{m+1}^*,\cdots,e_{n}^*\rangle.$$  We use multi-indices $I=\{ i_1, \ldots, i_p \} \subseteq \{0, \ldots, n \}$ with $i_1 < i_2 <  \ldots < i_p$ and we denote by $e_I^*$ the $p$-vector $e_{i_1}^* \wedge \ldots \wedge e_{i_p}^*$.  We also denote by $|I|$ the cardinality of $I$ and by $I^c$ the complement of $I$ in $\{0, \ldots, n \}$. A basis of $\F_p(0, \sigma)$ consists of $p$-vectors of the type
$$
e_I^*, \quad  \ I^c \cap \{0, \ldots, s \} \neq \emptyset.
$$
Similarly a basis of $\F_{p-1}(0, \sigma_{\infty})$ is given by the $p-1$-vectors
$$
e_J^*, \quad  J^c \cap \{1, \ldots, s \} \neq \emptyset.
$$
Then a basis of $\langle C(\tau) \rangle \wedge \mathcal{F}_{p-1}(0,\sigma_{\infty})$ is given by $p$-vectors of the type
$$
e_I^*,  \quad  I \cap \{0, m+1, \ldots, n \} \neq \emptyset \ \text{and} \ I^c \cap \{1, \ldots, s \} \neq \emptyset.
$$
This clearly shows that $\mathcal M_{p}^{\Delta}$ is free. 

Assume now that $\tau \in \partial T^\circ$ and $\sigma \in \partial T$ (in particular $\sigma = \sigma_{\infty}$). With the same notation as above, the quotient $N/ \langle C(\tau) \rangle$ can be identified with $\langle e_1^*,\cdots,e_m^*\rangle$.   By definition 
$$\mathcal Q_p (\hat{\tau}, \sigma) = \F_p(\hat{\tau}, \sigma)=\sum_{i=1}^{s} \bigwedge^p (e_i^\perp / \langle C(\tau) \rangle ), $$
so a basis of $Q_p(\hat{\tau}, \sigma)$ can be identified with $p$-vectors of the type
$$
e_I^*, \quad  I \subseteq \{1, \dots, m \} \ \text{and} \ I^c \cap \{1, \ldots, s \} \neq \emptyset.
$$
where now $I^c$ denotes the complement of $I$ in  $\{1, \dots, m \}$. 
Now notice that also a basis of 
\[ \mathcal Q_p (\tau, \sigma) = \mathcal M_p^{\Delta}(\tau, \sigma) =  \dfrac{\mathcal{F}_p(0,\sigma )}{\langle C(\tau) \rangle \wedge \mathcal{F}_{p-1}(0,\sigma)}\]
can be identified by the same set of $p$-vectors, since multi-indices $I$ such that $I \cap \{ 0, m+1, \ldots, n \} \neq \emptyset$ are killed by the quotient. This implies that $\mathcal Q_p(\tau, \sigma) \rightarrow \mathcal Q_p (\hat{\tau}, \sigma)$ is an isomorphism.  The case of $\mathcal R_p$ is trivial, indeed by definition $\mathcal{R}_p(\tau,\sigma) = \mathcal{R}_p(\tau, \hat{\sigma})$.

\end{proof}

We also have

\begin{cor} For any commutative ring $A$ we have short exact sequences
\[ \begin{split} 
     \ & 0 \rightarrow \mathcal M_p \otimes A \rightarrow \mathcal M_p^{\Delta} \otimes A \rightarrow \mathcal Q_p \otimes A \rightarrow 0, \\
\ & 0 \rightarrow \mathcal R_p \otimes A \rightarrow  \mathcal F_p  \otimes A \rightarrow \mathcal M_p^{\Delta} \otimes A \rightarrow 0, 
 \end{split}
 \] 
 and isomorphisms
 \[ \begin{split}  \mathcal Q_p(\tau, \sigma) \otimes A & \rightarrow \mathcal Q_p (\hat{\tau}, \sigma) \otimes A, \\
                        \mathcal R_p(\tau, \sigma) \otimes A & \rightarrow \mathcal R_p (\tau, \hat{\sigma} )\otimes A.
   \end{split}                \] 
\end{cor} 

\begin{proof} It follows from the fact that if we have a short exact sequence of $\Z$-modules where the last term is a flat module, then tensoring with any commutative ring preserves exactness. In our case the modules are free and therefore flat. 
\end{proof}

\begin{prop} \label{hom:QR} For all $k \geq 0$ and any commutative ring $A$ we have
\[ H_k (\mathcal{J}(T^\circ,T); \mathcal Q_p \otimes A) \cong 0, \] and 
\[ H_k (\mathcal{J}(T^\circ,T); \mathcal R_p \otimes A) \cong 0. \]
\end{prop}

\begin{proof}
For simplicity we denote $\mathcal Q_p \otimes A$ and  $\mathcal R_p \otimes A$ by $\mathcal Q_p$ and $\mathcal R_p$. Recall that the support of $\mathcal{Q}_p$ is 
$\mathcal{J}^{ub}(T^\circ,T)$. Introduce
$$
\mathcal{J}_0^{ub}=\{ (\tau, \sigma) \, | \, \sigma\in \partial T \ \text{and} \ \tau\in \partial T^\circ \}
$$
and recall that
$$
\mathcal{J}^\infty(T^\circ,T)=\{ (\tau, \sigma) \, | \, \sigma\in \partial T \ \text{and} \ 0\in \tau \}
$$
One has $\mathcal{J}^{ub}(T^\circ,T)=\mathcal{J}_0^{ub}\cup \mathcal{J}^\infty$, and there is a bijection between $\mathcal{J}_0^{ub}$ and $\mathcal{J}_\infty$ given by $(\tau,\sigma) \mapsto (\hat{\tau},\sigma)$. Denote by $\mathcal{C}^{ub}$ the complex of chains $\mathcal{C}(\mathcal{J}^{ub}(T^\circ,T),\mathcal{Q}_p)$, and similarly by $\mathcal{C}_0^{ub}$ the set of chains supported on $\mathcal{J}_0^{ub}$ and by $\mathcal{C}_\infty$ the complex $\mathcal{C}(\mathcal{J}^\infty (T^\circ,T),\mathcal{Q}_p)$. Notice that $\mathcal{C}_0^{ub}$ is not closed by taking the boundary. As a module we have a decomposition 
\[ \mathcal C^{ub} = \mathcal C_{0}^{ub} \oplus \mathcal C_{\infty} \]
and we denote by $\Pi_{0}$ and $\Pi_{\infty}$ the projections onto the respective component. 
Lemma \ref{q_iso} implies that 
\begin{equation} \label{pi_bound} 
    \Pi_{\infty} \circ \partial: \mathcal C_0^{ub} \rightarrow \mathcal C_{\infty}
\end{equation}
is an isomorphism.  Denote by $L: \mathcal C_{\infty} \rightarrow \mathcal C_0^{ub}$ its inverse. 
 Suppose now that $\gamma \in \mathcal C^{ub}$ is closed and decompose it as
 \[ \gamma = \gamma_0 + \gamma_{\infty}, \]
 its $\mathcal C_0^{ub}$ and $\mathcal C_{\infty}$ components. Then 
 \[ \partial \gamma_0 = - \partial \gamma_{\infty} \]
  Let 
 \[ \beta = L(\gamma_{\infty}) \]
 We claim that 
 \[ \partial \beta = \gamma. \]
First of all, decomposing $\partial \beta$ and using the fact that $L$ is the inverse of $\Pi_{\infty} \circ \partial$ we have
\[ \partial \beta = \Pi_{0}(\partial \beta) + \gamma_{\infty}. \]
Taking the boundary again, we obtain 
\[  \partial ( \Pi_{0} (\partial \beta)) = - \partial \gamma_{	\infty} = \partial \gamma_0, \]
i.e. 
\[ \partial ( \Pi_{0} (\partial \beta) - \gamma_0) = 0. \]
Now the fact that \eqref{pi_bound} is an isomorphism implies that 
\[  \Pi_{0} (\partial \beta) = \gamma_0. \]
Hence 
\[ \partial \beta = \gamma, \] and so  any closed chain in $\mathcal{C}^{ub}$ is exact.
The proof for $\mathcal{R}_p$ goes along the exact same lines, by considering the decomposition of $\mathcal{J}^0$ into $\mathcal{J}^{ub}_0 \cup \mathcal{J}^S$.
\end{proof}
Then, from the long exact sequences in homology associated to the above short exact sequences, we obtain canonical isomorphisms 
 \[ H_q (\mathcal{J}(T^\circ,T); \mathcal M_p \otimes A) \cong H_q (\mathcal{J}(T^\circ,T); \mathcal M^{\Delta}_p \otimes A) \cong  H_q (\mathcal{J}(T^\circ,T); \mathcal F_p \otimes A) \]
 for all $q \geq 0$. This proves \eqref{FtoM}. 
 
 \begin{ex} \label{inf:sph}
 Suppose $\gamma_{\infty} \in \mathcal{C}(\mathcal{J}^\infty (T^\circ,T),\mathcal{F}_p)$ is a cycle (supported at infinity) defining a class in $H_{\ast}(X_{T^\circ,T}, \mathcal F_p)$. Let 
 \[ \gamma_{\infty} = \sum_{(\tau, \sigma) \in \mathcal{J}^\infty (T^\circ,T)} (\tau, \sigma) \otimes v_{(\tau, \sigma)}, \]
 Consider the chain in $\mathcal{C}(\mathcal{J}^{ub}_{0} (T^\circ,T),  \mathcal M_p^{\Delta})$ given by
 \[ C =  \sum (\tau_{\infty}, \sigma) \otimes \tilde v_{(\tau, \sigma)}, \]
 where $\tilde v_{(\tau, \sigma)} \in \mathcal M_p^{\Delta}(\tau_{\infty}, \sigma)$ is the unique element corresponding to $v_{(\tau, \sigma)}$ under the isomorphism of Proposition \ref{q_iso}. Using the notation from Proposition \ref{hom:QR} we have $C = L (\gamma_{\infty})$. In particular 
 the boundary of $C$ is given by 
 \[\partial C = \gamma_{\infty} + \gamma_{S} \]
 where 
 \[ \gamma_{S} = \sum (\tau_{\infty}, \hat \sigma) \otimes \tilde v_{(\tau, \sigma)}\]
 is a cycle in $\mathcal{C}(\mathcal{J}^S (T^\circ,T),\mathcal{M}_p)$ defining the class $[\gamma_S] \in H_\ast( (\mathcal{J} (T^\circ,T), \mathcal M_p)$ corresponding to $[\gamma_{\infty}]$ under the isomorphism  \eqref{FtoM}.

 \end{ex}

Let us now prove \eqref{M_ms}.  It will be enough to prove it over $\Z$. Recall that the support of $\mathcal{M}_p$ is the subposet $\mathcal{J}^S(T^\circ,T)$, and that the bijection $j_{T^\circ,T}$ \eqref{mirror_subdiv} restricts to a bijection from $\mathcal{J}^S(T^\circ,T)$ to $\mathcal{J}^S(T,T^\circ)$ which sends $(\tau, \sigma) \in \mathcal{J}^S(T^\circ,T)$  to $(\sigma_{\infty}, \hat{\tau}) \in \mathcal{J}^S(T, T^\circ)$.  We now prove a stronger statement than  \eqref{M_ms}, namely that the pairs $(\mathcal{J}^S(T^\circ,T), \mathcal M_p)$ and $( \mathcal{J}^S(T, T^\circ) , \mathcal M_{n-p})$ are canonically isomorphic. By definition we have 
\[ \mathcal M_p(\tau,\sigma) = \frac{\mathcal{F}_p(0,\sigma)} { \langle C(\tau) \rangle \wedge \mathcal{F}_{p-1}(0,\sigma_{\infty})}\]
 and
 \[ \mathcal M_{n-p}(\sigma_\infty,  \hat{\tau}) =  \frac{\mathcal{F}_{n-p}(0,\hat{\tau})} { \langle C(\sigma) \rangle \wedge \mathcal{F}_{n-p-1}(0,\tau)}. \]
Given an orientation on $N_\R$, there exists a unique integral $(n+1)$-form $\Omega_{N}\in \bigwedge^{n+1} M$ compatible with the orientation. Given $z \in \mathcal{F}_p(0,\sigma)$, we denote by $[z]$ its class in $\mathcal M_p(\tau,\sigma)$. Let $v \in N$ be a vertex of $\tau$. Consider the map 
\begin{equation} \label{contr_map}
\begin{split}
 \mathcal M_p(\tau, \sigma) & \rightarrow \mathcal M_{n-p}(\sigma_\infty,  \hat{\tau}) \\
  [z] & \mapsto [\iota_{v \wedge z} \Omega_N ]
  \end{split}
\end{equation}
where $\iota$ denotes the contraction operator.  We have the following
\begin{prop}  After fixing an orientation of $N_\R$, for all $(\tau, \sigma) \in \mathcal J^S(T^\circ,T)$, the map \eqref{contr_map} is well defined and gives an isomorphism
 \[ \mathcal M_p(\tau, \sigma) \cong \mathcal M_{n-p}(\sigma_\infty,  \hat{\tau}) \] 
 independent of the choice of $v$. Moreover it is compatible with cosheaf maps, i.e. for any inclusion $(\tau_1,\sigma_1) \hookrightarrow (\tau,\sigma)$ the following diagram commutes 
\[ \begin{tikzcd}
  \mathcal M_p(\tau,\sigma) \arrow[r] \arrow[d] & \mathcal M_p(\tau_1,\sigma_1)  \arrow[d] \\
  \mathcal M_{n-p}(\sigma_\infty,\hat{\tau} )  \arrow[r] & \mathcal M_{n-p}((\sigma_1)_\infty,\hat{\tau}_{1}).
 \end{tikzcd} \]
 where the horizontal maps are the cosheaf maps. The same is true after tensoring with a commutative ring $A$.
\end{prop}

\begin{proof} Recall that $(\tau, \sigma) \in \mathcal{J}^S(T^\circ,T)$ implies that $\tau \in \partial T^\circ$ and $0 \in \sigma$. 
We will use the following essential properties of the contraction operator. First of all it satisfies the Leibniz property, i.e. for any vector $v$ and forms $\alpha$ and $\beta$ we have
\[ \iota_{v}(\alpha \wedge \beta) = (\iota_v \alpha) \wedge \beta + (-1)^{\deg \alpha} \alpha \wedge (\iota_v \beta).\]
Next, for any vector $v \in N$ we have
\begin{equation} \label{keriota}
   \bigwedge^p v^{\perp} = \left \{ \omega \in \bigwedge^p M\, | \, \iota_v \omega = 0  \right \}. 
\end{equation} 

Let us first prove that the map is well defined. Notice that
\[ \iota_v (\iota_{v \wedge z} \Omega_N) = 0 \]
In particular $ \iota_{v \wedge z} \Omega_N \in \bigwedge^{n-p} v^{\perp}$. Since $v$ is a vertex of $\tau$, it is also tangent to a one dimensional edge of $\hat \tau$. This proves that 
\[ \iota_{v \wedge z} \Omega_N \in \mathcal F_{n-p}(0, \hat \tau). \]
Now let $z \in  \langle C(\tau) \rangle \wedge \mathcal{F}_{p-1}(0,\sigma_{\infty})$. We need to show that $ \iota_{v \wedge z} \Omega_N \in \langle C(\sigma) \rangle \wedge \mathcal{F}_{n-p-1}(0,\tau)$. Let $f_1, \ldots, f_q \in M$ be the vertices of $\sigma_{\infty}$. We can complete them to form an oriented basis $\{f_1, \ldots, f_{n+1} \}$ of $M$, so that 
\[ \Omega_N = f_1 \wedge \ldots \wedge f_{n+1}. \]
We denote by $\{f_1^*, \ldots, f_{n+1}^* \}$ the dual basis of $N$. Let $e_1, \ldots, e_s \in N$ be the vertices of $\tau$. We can assume
\[ v = e_1. \]
Since $\sigma_\infty\subset \min(C(\tau))^\vee$ and by reflexivity we have 
\begin{equation}  \label{fe}
   \langle f_i, e_j\rangle = 1 \quad \text{for all} \ \ 1\leq i \leq q \ \text{and} \ 1 \leq j \leq s. 
\end{equation}
By definition 
\[ \F_p(0 , \sigma_{\infty})=\sum_{1 \leq i <j \leq q} \bigwedge^p (f_i - f_j)^\perp, \]
Without loss of generality we can consider 
\[ \omega \in \bigwedge^p (f_1 - f_2)^\perp. \]
and let 
\[ z = e_2 \wedge \omega. \]
Then $z \in  \langle C(\tau) \rangle \wedge \mathcal{F}_{p-1}(0,\sigma_{\infty})$. Observe that 
\[   (f_1 - f_2)^\perp = \langle f_1^* + f_2^*, f_3^*, \ldots, f_{n+1}^* \rangle.\]
Therefore, by linearity, it is enough to consider two cases for $\omega$: either $\omega = f^*_I$, with $I=\{i_1, \ldots, i_{p-1} \} \subset \{3, \ldots, n+1 \}$ or $\omega = (f_1^* + f_2^*) \wedge f_I^*$, with $I=\{i_1, \ldots, i_{p-2} \} \subset \{3, \ldots, n+1 \}$. In these two cases we compute $\iota_{e_1 \wedge z} \Omega_N = \iota_{e_1 \wedge e_2 \wedge \omega} \Omega_N$. 
In the first case we have
\[
    \iota_{e_1 \wedge e_2 \wedge \omega} \Omega_N  =  \iota_{e_1}( \iota_{e_2} (\iota _{f_I^*}  \Omega_N )) =  \iota_{e_1}( \iota_{e_2} (f_1 \wedge f_2 \wedge f_{I'} )) 
 \]
 where $I'$ is the complement of $I$ in $\{3, \ldots, n+1\}$. Using the Leibnitz rule for $\iota$ and relations \eqref{fe} we compute that
   \[ \begin{split}
    \ & \iota_{e_1}( \iota_{e_2} (f_1 \wedge f_2 \wedge f_{I'} )) = \iota_{e_1}((f_2 - f_1) \wedge f_{I'} +f_1 \wedge f_2 \wedge (\iota_{e_2} f_{I'})) =  \\
    \ &=(f_1 - f_2) \wedge (\iota_{e_1} f_{I'}) + (f_2 - f_1)  \wedge (\iota_{e_2} f_{I'}) + f_1 \wedge f_2 \wedge (\iota_{e_1 \wedge e_2} f_{I'})) = \\
   \ &= (f_1 - f_2) \wedge (\iota_{e_1-e_2} f_{I'}) + f_1 \wedge f_2 \wedge (\iota_{e_1 \wedge e_2} f_{I'})
   \end{split} \]
  We now have that $f_1$ and $f_1 - f_2$ are both in $C(\sigma)$ and $\iota_{e_1-e_2} f_{I'} \in \bigwedge^{n-p-1} (e_1 - e_2)^{\perp}$ by \eqref{keriota}. In particular, since $e_1 - e_2$ is tangent to a $1$-dimensional edge of $\tau$, we have $\iota_{e_1-e_2} f_{I'}  \in \mathcal F_{n-p-1}(0, \tau)$. Now observe that 
  \[ \iota_{e_1 - e_2} (f_2 \wedge (\iota_{e_1 \wedge e_2} f_{I'})) =  \iota_{e_1 - e_2} (f_2 \wedge (\iota_{(e_1 - e_2)\wedge e_2} f_{I'})) = 0. \] 
 So that also $f_2 \wedge (\iota_{e_1 \wedge e_2} f_{I'}) \in \mathcal F_{n-p-1}(0, \tau) $. Hence this proves, for this case, that $\iota_{v \wedge z} \Omega_N \in \langle C(\sigma) \rangle \wedge \mathcal{F}_{n-p-1}(0,\tau)$. 
 
 In the second case, i.e. $\omega = (f_1^* + f_2^*) \wedge f_I^*$, a similar computation shows that 
 \[ \iota_{e_1 \wedge e_2 \wedge \omega} \Omega_N = (f_2 - f_1) \wedge (\iota_{e_1 \wedge e_2} f_{I'}). \] 
 Again, also in this case $\iota_{v \wedge z} \Omega_N \in \langle C(\sigma) \rangle \wedge \mathcal{F}_{n-p-1}(0,\tau)$. This concludes the proof that \eqref{contr_map} is well-defined.
 
 Now consider two different vertices of $\tau$, $v$ and $v'$. We prove that $\iota_{v \wedge z}\Omega_N - \iota_{v' \wedge z} \Omega_N$ is in $\langle C(\sigma) \rangle \wedge \mathcal{F}_{n-p-1}(0,\tau)$. This would show that the map \eqref{contr_map} is independent of $v$. It is enough to check this for $z \in  \bigwedge^p f_1^{\perp}$. In particular we can assume $z = f_{I}^*$ for some multi-index such that $1 \notin I$. Then
  \[ \begin{split}
      \iota_{v \wedge z}\Omega_N - \iota_{v' \wedge z} \Omega_N & = \iota_{v-v'} ( \iota_{f^*_I} \Omega_N) =  \iota_{v-v'} (f_1 \wedge f_{I'}) \\
    = - f_1 \wedge (\iota_{v-v'} (f_{I'}))
   \end{split} \]
 for some multi-index $I'$ such that $1 \notin I'$. Now we have $f_1 \in \langle C(\sigma) \rangle$ and $\iota_{v-v'} (f_{I'}) \in \bigwedge^{n-p-1} (v-v')^\perp \subseteq \mathcal F_{n-p-1}(0, \tau)$. Hence $\iota_{v \wedge z}\Omega_N - \iota_{v' \wedge z} \Omega_N \in \langle C(\sigma) \rangle \wedge \mathcal{F}_{n-p-1}(0,\tau)$. 

We now prove that \eqref{contr_map} is invertible. In fact consider the same map, but in the opposite direction $\mathcal M_{n-p}(\sigma_\infty,  \hat{\tau})  \rightarrow  \mathcal M_p(\tau,\sigma)$, using a vertex $w$ of $\sigma_{\infty}$ (instead of $v$) and the $(n+1)$-form $\Omega_M\in\bigwedge^{n+1} N$. This map (up to a sign) is the inverse of \eqref{contr_map}. Indeed we prove that 
\[ [\iota_{w \wedge (\iota_{v \wedge z} \Omega_{N})} \Omega_M] = \epsilon [z] \]

where $\epsilon= (-1)^{\tfrac{n(n+5)}{2}}$. 

It is enough to consider $z \in \bigwedge^{p} f_1^{\perp}$, in particular we can assume 
\[z = f_I^{*} \]
where $I$ is some multi-index with $1 \notin I$ and $|I| = p$. Define the following sign
\[ (-1)^{I} := \prod_{k \in I} (-1)^{k+1} \]
Then 
\[ \iota_{v \wedge z} \Omega_{N} = (-1)^{I} \iota_v(f_1 \wedge f_{I'}) =  (-1)^I (f_{I'} - f_1 \wedge (\iota_v f_{I'}))\]
where $I'$ is the complement of $I$ in $\{ 2, \ldots, n+1 \}$. We can assume that the vertex $w$ of $\sigma_{\infty}$ is $f_1$. Hence 
\[ w \wedge (\iota_{v \wedge z} \Omega_{N}) = (-1)^I f_1 \wedge f_{I'}. \]
We have 
\[ \Omega_M = f_1^* \wedge \ldots \wedge f_{n+1}^*. \]
Therefore 
\[ \iota_{w \wedge (\iota_{v \wedge z} \Omega_{N})} \Omega_M = (-1)^I \iota_{f_1 \wedge f_{I'}} \Omega_M = (-1)^I (-1)^{I'} f_{I}^* = (-1)^I (-1)^{I'} z. \]
Since $I \cup I' = \{2, \ldots, n+1 \}$, it is easy to show that $ (-1)^I (-1)^{I'} = (-1)^{\tfrac{n(n+5)}{2}}$. This concludes the proof that \eqref{contr_map} is an isomorphism. 

Let us prove the compatibility with cosheaf maps. We have an inclusion $(\tau_1, \sigma_1) \hookrightarrow (\tau, \sigma)$ if and only if $\tau \subset \tau_1$ and $\sigma \subset \sigma_1$. In particular in defining \eqref{contr_map},  the vertex $v$ of $\tau$ is also a vertex of $\tau_1$. Moreover the coheaf maps $\mathcal F_p(0, \sigma) \rightarrow \mathcal F_p(0, \sigma_1)$ and $\mathcal F_{n-p}(0, \hat \tau) \rightarrow \mathcal F_{n-p}(0, \hat \tau_1)$ are just inclusions of sub-modules. So that we have a commuting diagram 
\[ \begin{tikzcd}
  \mathcal F_p(0,\sigma) \arrow[r] \arrow[d] & \mathcal F_p(0,\sigma_1)  \arrow[d] \\
  \mathcal F_{n-p}(0,\hat{\tau} )  \arrow[r] & \mathcal F_{n-p}(0,\hat{\tau}_{1}).
 \end{tikzcd} \]
where the vertical arrows are given by the map $z \mapsto \iota_{v \wedge z} \Omega_N$. When we pass to the quotients we get the commuting diagram in the statement of the theorem.
 \end{proof}
 
In particular this proves that the pairs $(\mathcal{J}^S(T^\circ,T), \mathcal M_p \otimes A)$ and $( \mathcal{J}^S(T, T^\circ) , \mathcal M_{n-p} \otimes A)$ are canonically isomorphic and therefore they have canonically isomorphic homologies.  This concludes the proof.

\section{Patchworking and mirror symmetry}
We now consider real Calabi-Yau hypersurfaces coming from combinatorial patchworking. We will use Theorem \ref{thm:main1} in order to give a necessary and sufficient condition for such real Calabi-Yau to be connected.
\subsection{Combinatorial patchworking}

The patchworking method was invented by Viro in the 1980's and it has proved to be a very powerful tool to construct real algebraic hypersurfaces (and complete intersections) in toric varieties with prescribed topology of the real part \cite{viro_patch}, \cite{viro}. Let us recall here the (primitive) combinatorial version. Let $\Delta$ be a lattice polytope in $M_\R$ and $T$ a primitive lattice triangulation of $\Delta$. Primitive (or unimodular) means here that all facets of $T$ have lattice volume $1$. Let $\varepsilon:\Delta\cap M\rightarrow \mathbb{F}_2$ be a sign distribution on the lattice points of $\Delta$ (since the triangulation is primitive, note that the lattice points of $\Delta$ coincide with the vertices of the triangulation). When we say signs in the context of Viro's patchworking, we will identify $+$ with $0\in\FF_2$ and $-$ with $1$.

Recall that the real part of the toric variety associated to $\Sigma_\Delta$ is constructed by considering a certain quotient of copies $\Delta(\xi)$ of $\Delta$ for any group homomophism $\xi:M \rightarrow \mathbb{F}_2$, see Equation (\ref{eq:real part}).
Denote by $T(\xi)$ the triangulation of $\Delta(\xi)$ induced by $T$.  Given a lattice point $v\in \Delta$, the sign of its copy $v(\xi)$ in $\Delta(\xi)$ is defined by 
$\varepsilon(v(\xi)) = \xi(v)+\varepsilon(v) \in \mathbb{F}_2$. For any $\xi$ and any $(n+1)$-simplex $\sigma$ of $\Delta(\xi)$, separate the vertices having different signs by taking the union of convex hulls of barycenters of simplices $\sigma_1\subset\cdots\subset \sigma$ such that $\sigma_1$ is an edge with different signs. This gives a simplicial complex defining a PL-hypersurface in the quotient $\R\Sigma_\Delta$, denoted by $X_{T,\varepsilon}$. See Figures \ref{pw_divisors} and \ref{pw_divisors_2} for examples (ignore the divisors for the moment). In the standard description of combinatorial patchworking \cite{itenberg1995}, the hypersurface $X_{T,\varepsilon}$ is defined by considering convex hulls of midpoints of edges with different signs, but the two descriptions give the same result up to isotopy. In the case of a convex triangulation induced by a convex function $\mu: \Delta \rightarrow \R$, Viro's combinatorial patchworking theorem states that the PL-hypersurface $X_{T,\varepsilon}$ is ambient isotopic (in $\R\Sigma_\Delta$) to the closure (in $\R\Sigma_\Delta$) of the hypersurface in $(\R^*)^{n+1}$ defined by the polynomial 
\begin{equation}\label{Viro polynomial}
\sum_{v\in \Delta\cap M} (-1)^{\varepsilon(v)} t^{\mu(v)} x^v,
\end{equation}
for $t$ small enough and positive. 
In this text, we consider also different compactifications. In fact, given unimodular central triangulations $T^{\circ}$ and $T$ of a pair of dual reflexive polytopes $\Delta^{\circ}$ and $\Delta$, we consider the toric variety coming from the fan $\Sigma_{T^\circ}$, where $T^\circ$ is a unimodular triangulation of $\Delta^\circ$. We denote by $X_{T^\circ,T,\varepsilon}$  the compactification of $X_{T,\varepsilon}\cap (\R^*)^{n+1}$ in $\R\Sigma_{T^\circ}$. Again, under convexity hypothesis, the PL-hypersurface $X_{T^\circ,T,\varepsilon}$ is ambient isotopic in $\R\Sigma_{T^\circ}$ to the compactification in $\R\Sigma_{T^\circ}$ of the hypersurface given by \eqref{Viro polynomial}.

\subsection{Patchworking and divisors in the mirror}

Consider the Picard group of a toric variety $\C\Sigma$ with $\FF_{2}$ coefficients, denoted by $\pic_{\FF_2}(\C\Sigma)$.  
For a smooth toric variety $\C\Sigma$ with $\Sigma$ a complete fan in $M$, this group is computed as follows. Let $\dv_{\text{toric}} \C\Sigma$ be the free module generated by toric divisors, i.e. 
\[ \dv_{\text{toric}} \C\Sigma = \sum_{\rho \in \Sigma(1)} \Z_{2} D_{\rho} \] 
where the sum runs over all rays of the fan (i.e. $1$-dimensional cones) and $D_{\rho}$ is the toric divisor associated to the ray. Then we have the following short exact sequence
\begin{equation} \label{2div}
     0 \rightarrow N \otimes \FF_2 \rightarrow \dv_{\text{toric}} \C\Sigma \rightarrow \pic_{\FF_2}(\C\Sigma) \rightarrow 0.
\end{equation}

The second arrow is the map $n \mapsto \sum_{\rho} \langle n,v_{\rho} \rangle D_{\rho}$, where $v_{\rho} \in M \otimes \FF_{2}$ is the $\mod 2$ reduction of the primitive generator of the ray $\rho$, see \cite{fulton:toric} and \cite{cox:little:schenck}.

 We now go back to the dual reflexive polytopes $\Delta$ and $\Delta^{\circ}$ with unimodular central triangulations $T$ and $T^{\circ}$. Suppose we do patchworking with $T$ and some sign distribution $\varepsilon$, thus obtaining a PL-hypersurface $X_{T^{\circ}, T,\varepsilon}$ in $\R\Sigma_{T^\circ}$, which , in the convex case, is isotopic to a real Calabi-Yau variety.   

The sign distribution $\varepsilon$ also determines a toric divisor $D^{\varepsilon} \in \dv_{\text{toric}} \C\Sigma_T $ as follows. Let $\mathbf{o} \in T$ be the unique interior lattice point, then the other lattice points $v \in \partial T$ are the primitive generators of the rays of the fan $\Sigma_T$. Therefore we can define
\[ D^{\varepsilon} = \sum_{\varepsilon(v) = \varepsilon (\mathbf o)} D_{v}. \]
It is clear that, up to inverting the sign of every lattice point in $T$, there is a one to one correspondence between toric divisors with $\FF_2$ coefficients on $\C\Sigma_T$ and a choice of patchworking signs on $T$. Therefore we denote by $X_{T^{\circ}, T, D}$  the PL-hypersurface in $\R\Sigma_{T^{\circ}}$ constructed by the patchworking signs determined by $D \in \dv_{\text{toric}} \C\Sigma_T$. We have then the following:

\begin{prop} \label{div:pw} If two divisors $D$ and $D'$ in $\dv_{\text{toric}} \C\Sigma_T$ define the same class in $\pic_{\FF_2}(\C\Sigma_T)$ then 
$X_{T^{\circ}, T, D}$ and $X_{T^{\circ}, T, D'}$ coincide, up to an automorphism of $\R\Sigma_{T^{\circ}}$.
\end{prop}

\begin{proof} We will prove that $D$ and $D'$ are equivalent $\FF_2$ divisors if and only if they correspond to sign distributions $\varepsilon$ and $\varepsilon'$ satisfying $\varepsilon'=\varepsilon+\xi$ for some $\xi\in \Hom(M,\mathbb{F}_2)$. This obviously implies the statement. Using \eqref{2div}, $D$ and $D'$ are equivalent divisors if and only if 
\[ D' = D + \sum_{v \in \partial T} \langle n, v - \mathbf o \rangle D_v. \]
for some $n \in N \otimes \FF_2$. Note that $N\otimes \FF_2=\Hom(M,\FF_2)$.

Let 
\[ D = \sum_{v \in \partial T} a_v D_v \]
with $a_v \in \FF_2$. 

Fix the sign of $\mathbf o$ to be $\varepsilon (\mathbf o) = 1\in\FF_2$. Then the signs corresponding to $D$ are $\varepsilon (v) = a_v$. 
Let $\xi\in \Hom(M,\FF_2)$. We see that for $v\in\partial T$, one has
$$
\varepsilon(v)+\xi (v) = \varepsilon(\mathbf o) + \xi (\mathbf o)
$$
iff 
$$
a_v+\xi (v-\mathbf o)=1.
$$
This means that if $\xi = n$, the divisor corresponding to $\varepsilon+\xi$ is precisely 
$$
D + \sum_{v \in \partial T} \langle n, v - \mathbf o \rangle D_v,
$$
which is $D'$. This concludes the proof of the lemma.

\end{proof}

\begin{ex} The real cubic curves in $\PP^2$ obtained from patchworking with the sign distribution corresponding to $D= D_7 + D_8$ and to $D=D_8$  are depicted in Figure \ref{pw_divisors}. Notice that the first one is connected, while the second one is not.

\begin{figure}[!ht] 
\begin{center}
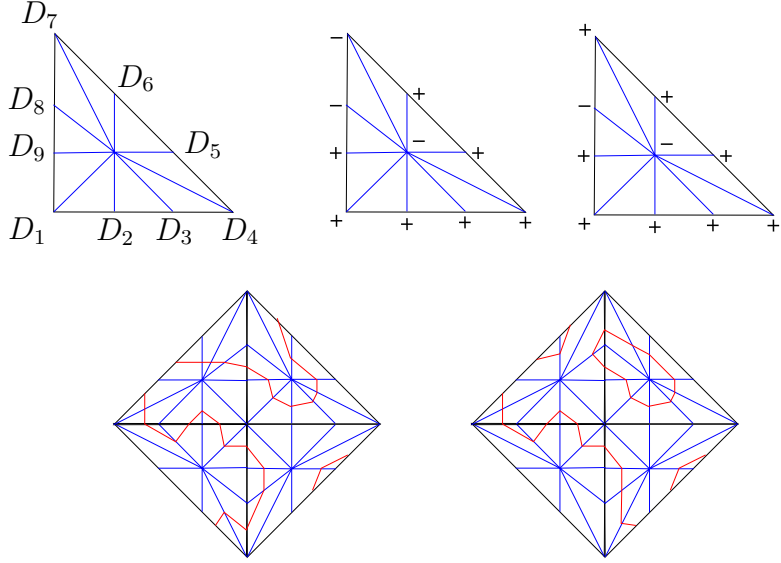
\caption{The signs corresponding to $D = D_7 + D_8$ and to $D=D_8$ and the corresponding real cubics.}
\label{pw_divisors}
\end{center}
\end{figure}
\end{ex}

\subsection{Real structures and the patchworking spectral sequence} \label{rps} In this section, we describe in our context the spectral sequence from \cite{ren_shaw:bound} allowing to understand better the Betti numbers of a combinatorial patchworking. 
\begin{defi} A \emph{real phase structure} on $(\Delta, T)$  is the data $\mathcal E$ given by a choice, for every edge $\sigma \in T$, of an element  $\mathcal E_{\sigma} \in \frac{N \otimes \FF_2}{\sigma^{\perp}}$, such that for any two dimensional simplex $\tau \in T$ the number of edges $\sigma$ of $\tau$ such that $\mathcal E_{\sigma} = 0$ is either $0$ or $2$.
\end{defi}

This definition appeared first in \cite{ren_shaw:bound} for tropical hypersurfaces. Notice that $\mathcal E_{\sigma}$ can also be interpreted as a choice of an affine hyperplane of $N \otimes \FF_2$ parallel to $\sigma^{\perp}$. The condition on the two dimensional simplex $\tau$ is that the three affine hyperplanes, corresponding to the edges of $\tau$, pairwise intersect in three distinct codimension $2$ subspaces parallel to $\tau^{\perp}$. A real phase structure is equivalent to a choice of patchworking signs on the triangulation $T$, up to inversion of all signs. Indeed this follows by establishing the rule that the vertices of an edge $\sigma$ have the same sign if and only if  $\mathcal E_{\sigma}$ is the generator of $\frac{N \otimes \FF_2}{\sigma^{\perp}}$.  Then a sign distribution determines a real phase structure, viceversa, given a real phase structure $\mathcal E$, the choice of the sign of one vertex determines the signs of all other vertices. 
 
\begin{defi}
Let $\mathcal E$ be a real phase structure on $(\Delta, T)$. For $(\tau,\sigma)\in\mathcal{P}^1(T^\circ,T)$, define
$$
\mathcal{E}(\tau,\sigma)=\left\lbrace s\in N\otimes \FF_2 / \langle C(\tau) \rangle \:\: \vert \:\: \pi_{\gamma}(s)=\mathcal{E}_\gamma \text{ for some edge } \gamma\subset\sigma \right\rbrace,
$$
where $\pi_{\gamma}:  N\otimes \FF_2 /  \langle C(\tau) \rangle \rightarrow  N\otimes \FF_2 / \gamma^\perp$. 
\end{defi}
Define also 
$$
\mathcal{P}_\R(T^\circ,T)=	\left\lbrace (\tau, \sigma, s) \:\:\vert \:\: (\tau,\sigma)\in \mathcal{P}(T^\circ,T), s\in N\otimes \FF_2 / C(\tau)  \right\rbrace.
$$

The set $\mathcal{P}_\R(T^\circ,T)$ is also a poset, where the order is given by $(\tau, \sigma, s)\leq (\tau', \sigma', s')$ if
$(\tau, \sigma)\leq_\mathcal{P} (\tau', \sigma')$ and $\pi_{\tau',\tau}(s')=s$, where $\pi_{\tau',\tau}:N\otimes \FF_2 /  \langle C(\tau') \rangle \rightarrow N\otimes \FF_2 / \langle C(\tau) \rangle$ is the $\FF_2$ reduction of the projection maps (\ref{cosheafmaps}). Consider the geometric realization of the poset $\mathcal{P}_\R(T^\circ,T)$ and its subdivision given by 
$$
\vert \mathcal{P}_\R(T^\circ,T)\vert = \bigcup_{(\tau,\sigma, s)\in \mathcal{P}_\R(T^\circ,T)} F(\tau,\sigma,s).
$$
where
$$
F(\tau,\sigma,s)= \bigcup_{(\tau,\sigma,s)_\bullet \in (\tau,\sigma,s)_{\leq}}\vert (\tau,\sigma,s)_\bullet \vert.
$$
\begin{prop}
The poset $\mathcal{P}_\R(T^\circ,T)$ is the face poset of a regular CW-structure on $\R\Sigma_{T^\circ}$ with cells $F(\tau,\sigma,s)$.
\end{prop}
\begin{proof}
This is similar to Proposition \ref{prop:tropicaltoricposet}. Here the face poset of the real toric variety associated to a fan $\Sigma$ is  the poset
$$
\Sigma_\R:=\left\lbrace (\rho,s) \: \vert \: \rho\in\Sigma \text{ and } s\in N\otimes\FF_2 / \rho \right\rbrace.
$$
Again the poset $\mathcal{P}_\R(T^{\circ},T)$ defines a subdivision of $\Sigma_{T^\circ,\R}$ so their geometric realizations are homeomorphic.
\end{proof}
Given a real phase structure $\mathcal{E}$ on $(\Delta,T)$, consider the following subposet of $\mathcal{P}_\R(T^{\circ}, T)$:
$$
\mathcal{P}^1_\mathcal{E}(T^\circ,T)=	\left\lbrace (\tau, \sigma, s) \in \mathcal{P}_\R(T^\circ,T) \:\: \vert \:\: \dim \sigma \geq 1, \ \ s\in\mathcal{E}(\tau,\sigma) \right\rbrace.
$$

\begin{prop}
Let $\mathcal{E}$ be a real structure on $(\Delta,T)$ and let $\varepsilon:\Delta\cap M\rightarrow \FF_2$ be a corresponding sign distribution. Then $\mathcal{P}^1_\mathcal{E}(T^\circ,T)$ is a regular CW-structure on $X_{T^\circ,T,\varepsilon}$.
\end{prop}

\begin{proof}
Recall that the hypersurface $X_{T,\varepsilon}$ is obtained by considering for any $\xi\in N\otimes\FF_2$ and any $(n+1)$-simplex $\sigma$ of $\Delta (\xi)$ the union of convex hulls of barycenters of simplices $\sigma_1\subset\cdots\subset \sigma$ such that $\sigma_1$ is an edge with different signs. So after taking the quotient in $\R\Sigma_\Delta$, the hypersurface $X_{T,\varepsilon}$ is realized by the poset
$$
\left\lbrace ((\sigma_1\subset\cdots\subset\sigma_{k}),\xi) \:\: \vert \:\: \sigma_i\in T, \xi\in N\otimes\FF_2/(F_k)^\perp, \pi_{\sigma_1}(\xi)=\mathcal{E}_{\sigma_1} \right\rbrace.
$$
Here $F_k$ denotes the minimal face of $\Delta$ containing $\sigma_k$. Recall also that the poset $\mathcal{P}(T^\circ,T)$ was defined as a small modification of the poset
$$
\Phi:=\left\lbrace (\rho,\sigma)\in\Sigma_\Delta\times T \mid \sigma\subset \rho^\vee \right\rbrace.
$$
One can consider
$$
\Phi_\mathcal{E}:=\left\lbrace (\rho,\sigma,s) \: \vert \: (\rho,\sigma) \in \Phi \text{ and } s\in \mathcal{E}(\rho,\sigma) \right\rbrace.
$$
The surjective map of posets 
$$
 ((\sigma_1\subset\cdots\subset\sigma_{k}),\xi)\rightarrow (F_k^\vee,\sigma_1,\xi).
$$
shows that $\Phi_\mathcal{E}$ induces a regular CW-structure on $X_{T,\varepsilon}$. The result is obtained by taking the preimages by the map $\operatorname{Bl}: \T\Sigma_{T^\circ} \rightarrow \T\Sigma_{\Delta}$ as in the proof of Proposition \ref{prop:tropicaltoricposet}.
\end{proof}

There is an obvious surjective map from $\mathcal{P}_\R(T^\circ,T)$ to $\mathcal{P}(T^\circ,T)$ by forgetting the last coordinate, and results from \cite{ren_shaw:bound} extend to this case. We recall them for reader convenience. The proofs are similar as in \cite{ren_shaw:bound}.
\begin{defi}
The sign cosheaf on $\mathcal{P}(T^\circ,T)$ is given by 
$$
\mathcal{S}(\tau,\sigma):=\FF_2^{\mathcal{E}(\tau,\sigma)}.
$$
For $\varepsilon\in \mathcal{E}(\tau,\sigma)$, we will denote by $w_\varepsilon$ the corresponding generator of $\mathcal{S}(\tau,\sigma)$.
\end{defi}

\begin{prop}\label{prop:realposethom}
For every integer $q$, the groups $H_q(X_{T^\circ,T}; \mathcal{S})$ and $H_q(X_{T^\circ,T,\varepsilon};\FF_2)$ are isomorphic.

\end{prop}

\begin{prop}\label{prop:filtration}
Let $\mathcal{E}$ be a real structure on $(\Delta,T)$. There exists a filtration of cosheaves on $\mathcal{P}(T^\circ,T)$
$$
0=\mathcal{K}_{n+1}\subset\mathcal{K}_{n}\subset\mathcal{K}_{n-1}\subset\cdots\subset \mathcal{K}_0=\mathcal{S}
$$
such that 
$$
\mathcal{K}_{p}/\mathcal{K}_{p+1}\simeq \mathcal{F}_p.
$$
\end{prop}
For the reader's convenience, we recall briefly the definition of the filtration $\mathcal{K}_p$. If $\dim \sigma =1$, then $\mathcal{E}(\tau,\sigma)$ is an affine space of direction $\sigma^\perp / \langle C(\tau) \rangle $. Define in this case 
$$\mathcal{K}_p(\tau,\sigma)=\FF_2 \langle \sum_{\varepsilon\in H}w_\varepsilon \:\vert \: H\in \aff_p(\mathcal{E}(\tau,\sigma))\rangle,$$ where $\aff_p(\mathcal{E}(\tau,\sigma))$ is the space of affine subspaces of $\mathcal{E}(\tau,\sigma)$ of dimension $p$. In the general case, consider the sum
$$
\mathcal{K}_p(\tau,\sigma):=\sum_{\gamma\subset\sigma \: \dim\gamma=1}\mathcal{K}_p(\tau,\gamma).
$$
The spectral sequence $(E_{p,q}^{k}, \delta^{[k]})$ associated to this filtration has first page $$E^1_{p,q}=H_q(X_{T^\circ,T};\mathcal{F}_p)$$
and converges to $H_\bullet(X_{T^\circ,T};\mathcal{S})$ which by Proposition \ref{prop:realposethom} is the homology of $X_{T^\circ,T,\varepsilon}$.

\section{Connectedness} \label{connect}
To simplify a little the notations in this section, for a $\Z$-cosheaf $\mathcal{G}$, we will denote by $\mathcal{G}^{\FF_2}$ the tensor product $\mathcal{G}\otimes \FF_2$.
Let   $D=\sum_{\rho\in\Sigma_T(1)}a_\rho D_\rho$ be a toric divisor on $\C\Sigma_{T}$ with $\FF_2$-coefficients. Denote by $\supp(D)$ the support of $D$, meaning the set of rays in $\Sigma_T$ with a non-vanishing coefficient. Again, to simplify the notations we will identify the ray $\rho$ with the segment $S(\rho)$ in $T$. The divisor $D$ defines, by the same formula, a divisor on $\T\Sigma_T$ and when restricted to $X_{T,T^\circ}$ it defines a class in $H_{n-1}(X_{T,T^\circ};\mathcal{F}_{n-1}^{\FF_2})$ represented by the cycle
$$
D_{|X_{T,T^{\circ}}} :=\sum_{
\tiny{\begin{array}{c}
\rho\in \supp(D) \\ \dim(\delta)=1 \text{ and } \delta \subset \min(\rho)^\vee
\end{array}}} (\rho,\delta)\otimes v_{\rho,\delta},
$$
where $v_{\rho,\delta}$ is the generator of $\Lambda_{\FF_2}^{n-1}(\frac{\delta^\perp}{\rho})$. 
\begin{rem}
If $T$ and $T^\circ$ are convex, the Viro polynomial defines the mirror family of complex hypersurfaces $X_t^{\circ}$ in $\C\Sigma_T$, and the divisor $D$ restricts to a family of divisors $D_{|X_t^{\circ}}=\sum_{\rho\in\Sigma_T(1)}a_\rho D_\rho\cap X^{\circ}_t$. Each $D_\rho\cap X^{\circ}_t$ when restricted to the toric orbit $\mathcal{O}^\C_\rho$ defines in turn a Viro polynomial which is the one of the tropical hypersurface $X_{T,T^\circ}\cap \mathcal{O}^\T_\rho$. The facets of $X_{T,T^\circ}\cap \mathcal{O}^\T_\rho$ are precisely dual to the edges $\delta$ of $\Delta^\circ$ such that $\delta \subset \min(\rho)^\vee$.
\end{rem}
On the other hand the divisor $D$ also defines a real phase structure $\mathcal{E}=\mathcal{E}^D$ on $(\Delta,T)$ and a corresponding patchworking $X_{T^\circ,T,\varepsilon}$.  In this section we will prove the following: 

\begin{thm} \label{connectedness}
Assume that 
\begin{equation}\label{vanishinghypothesis}
H_n(X_{T^\circ,T};\mathcal{F}^{\FF_2}_k)=0 \:\: \text{ for all }\:\: 0<k<n.
\end{equation}
 Then the patchworking $X_{T^\circ,T,\varepsilon}$ is connected if and only if $[D_{|X_{T,T^{\circ}}}] = 0$ in $H_{n-1}(X_{T,T^\circ};\mathcal{F}_{n-1}^{\FF_2})$. 
\end{thm} 
\begin{rem}
The hypothesis (\ref{vanishinghypothesis}) is satisfied if $\Delta$ is a  polytope giving rise to a non-singular toric variety. This follows from Proposition 3.2 in \cite{brugalle2022combinatorial} (see also \cite{lefschetz:ARS}).
We expect that in the convex case, the tropical divisor class $[D_{|X_{T,T^{\circ}}}]$ vanishes in $H_{n-1}(X_{\text{trop}}^\circ;\mathcal{F}_{n-1}^{\FF_2})$ if and only if, for $t$ small enough, the class of the complex divisor $[D_{|X^{\circ}_t}]$ vanishes in the Picard group modulo $2$ of $X^{\circ}_t$.
\end{rem}
Recall that the real phase structure $\mathcal{E}^D$ defines a sign cosheaf $\mathcal{S}_D$ on $\mathcal{P}(T^\circ,T)$  and induces a spectral sequence $(E^k_{p,q}, \delta^{[k]})$ converging to $H_*(X_{T^\circ,T}; \mathcal{S}_D)$. Obviously the spectral sequence depends on $D$, but to keep notation simple we remove $D$ from the decorations.  In particular the number of connected components of $X_{T^\circ,T,\varepsilon}$ is given by the dimension of the homology group
\[ H_{n}(X_{T^\circ,T,\varepsilon}, \FF_2) \cong \bigoplus_{k} E^{\infty}_{k,n}. \]
Therefore, with assumption \eqref{vanishinghypothesis}, the number of connected components of $X_{T^\circ,T,\varepsilon}$ can be either one or two.   If the differential at the first page
\[ \delta^{[1]}: H_n(X_{T^\circ,T}; \mathcal{F}^{\FF_2}_0) \rightarrow  H_{n-1}(X_{T^\circ,T}; \mathcal{F}^{\FF_2}_1) \] 
does not vanish, i.e. it is injective, then $X_{T^\circ,T,\varepsilon}$ is connected. If it vanishes we need to study also higher differntials $\delta^{[k]}$. Recall that $X_{T^\circ, T}$ contains the sphere $S = S_{T^\circ, T}$ (see \eqref{trop:sphere}) and it is homotopically equivalent to it. One direction of Theorem \ref{connectedness} follows from 

\begin{thm}\label{thm:connectedness}
Let $S\in H_n(X_{T^\circ,T}; \mathcal{F}^{\FF_2}_0)$ be the generator. Then $\delta^{[1]}(S)$ and $[D_{|X_{T,T^{\circ}}}]$ are mirror classes. In particular, if  $[D_{|X_{T,T^{\circ}}}] \neq 0$ then $X_{T^\circ,T,\varepsilon}$ is connected. 
\end{thm}
\begin{proof}  We work with the refinement $\mathcal{J}(T^\circ,T)$ of $\mathcal{P}(T^\circ,T)$. The generator of $H_n(\mathcal{J}(T^\circ,T); \mathcal{F}^{\FF_2}_0)$ is given by
$$
S:= \sum_{
\tiny{\begin{array}{c}
(\tau,\sigma)\in\mathcal{J}^S(T^\circ,T) \\ \dim(\tau)=0 \text{ and } \dim(\sigma)=1
\end{array}}} (\tau,\sigma)\otimes 1.
$$
Then a lift of $S$ in $C_n(\mathcal{J}(T^\circ,T);\mathcal{S}_{D})$ is given by $\tilde{S}=S_0+S_1$ with
$$
S_0=\sum_{\rho\notin\text{supp}(D)}\sum_{
\tiny{\begin{array}{c}
(\tau,\rho )\in\mathcal{J}(T^\circ,T) \\ \dim(\tau)=0
\end{array}}} (\tau,\rho)\otimes w_{0},
$$
and
$$
S_1:=\sum_{\rho\in\text{supp}(D)}\sum_{
\tiny{\begin{array}{c}
(\tau, \rho)\in\mathcal{J}(T^\circ,T) \\ \dim(\tau)=0
\end{array}}} (\tau, \rho)\otimes w_{\tau}.
$$
The boundary of the chain $\tilde{S}$ in $C_{n-1}(\mathcal{J}(T^\circ,T);\mathcal{K}_1^D)$ is given by $A+B$ with
$$
A=\sum_{\rho\in\text{supp}(D)}\sum_{
\tiny{\begin{array}{c}
(\delta, \rho )\in\mathcal{J}(T^\circ,T) \\ \dim(\delta)=1
\end{array}}} (\delta,\rho)\otimes (w_{\tau_1}+w_{\tau_2}),
$$
where $\tau_1$ and $\tau_2$ are the two vertices of $\delta$, and
$$
B=\sum_{
\tiny{\begin{array}{c}
(\tau,\lambda)\in\mathcal{J}^S(T^\circ,T) \\ \dim(\tau)=0\text{ and } \dim(\lambda)=2
\end{array}}} (\tau,\lambda)\otimes (w_{(\tau,\lambda,\rho_1)}+w_{(\tau,\lambda,\rho_2)}),
$$
where $\rho_1$ and $\rho_2$ are the two rays of $\Sigma_T$ adjacent to $C(\lambda)$, and 
$$
w_{(\tau, \lambda,\rho_i)}=\left\lbrace \begin{array}{c} w_0 \text{ if } \rho_i\notin \text{supp}(D), \\
w_{\tau} \text{ if } \rho_i\in \text{supp}(D).
\end{array}\right.
$$
So finally, the image of $S$ by the map $\delta^{[1]}$ is represented by $\gamma_1+\gamma_2$, with 
\[
\gamma_1=\sum_{\rho\in\text{supp}(D)}\sum_{
\tiny{\begin{array}{c}
(\delta,\rho)\in\mathcal{J}(T^\circ,T) \\ \dim(\delta)=1
\end{array}}} (\delta, \rho)\otimes (\tau_1 - \tau_2)
\]
 and
$$
\gamma_2=\sum_{
\tiny{\begin{array}{c}
(\tau,\lambda)\in\mathcal{J}(T^\circ,T) \\ \dim(\tau)=0\text{ and } \dim(\lambda)=2
\end{array}}} (\tau,\lambda)\otimes v_{(\tau,\lambda)},
$$
where $v_{(\tau,\lambda)}=\tau$ if only one of the two rays of $\Sigma_T$ adjacent to $C(\lambda)$ is contained in the support of $D$, and is equal to $0$ otherwise. By definition of the cosheaves $\mathcal{M}_p$, the chain $\gamma_2$ is equal to $0$ in $C_{n-1}(\mathcal{J}(T^\circ,T); \mathcal{M}_1^{\FF_2})$. So in $H_{n-1}(\mathcal{J}(T^\circ,T); \mathcal{M}_1^{\FF_2})$, one has
$$
\delta^{[1]}(S)=\left[\gamma_1\right].
$$
Now, the mirror class of $\left[\gamma_1\right]$ in $H_{n-1}(\mathcal{J}(T,T^\circ); \mathcal{M}_{n-1}^{\FF_2})$ is given by
$$
\check{\gamma}_1= \sum_{
\tiny{\begin{array}{c}
\rho\in\text{supp}(D) \\ \dim(\delta)=1 \text{ and } \delta \subset \min(\rho)^\vee
\end{array}}} (\rho_\infty,\hat{\delta})\otimes \iota_{v\wedge v_\delta}\Omega_N,
$$
where $v$ is a vertex of $\delta$ and we have set $v_{\delta} = \tau_1 -\tau_2$, the direction of $\delta$. Notice that the submodule generated by $v$
and $v_\delta$ is equal to $\langle \hat{\delta} \rangle$ and so 
$ \iota_{v\wedge v_\delta}\Omega_N$ is the generator of  $ \Lambda^{n-1}  \langle \hat{\delta} \rangle ^\perp$. In particular $ \iota_{v\wedge v_\delta}\Omega_N$ can also be seen as a generator of $(\mathcal M_{n-1}^{\Delta})^{\FF_2}(\rho_{\infty}, \delta)$. 
Therefore, as in Example \ref{inf:sph}, the chains $\check{\gamma}_1$ and $D_{|X_{T,T^{\circ}}}$ are homologous by the following chain in $C_{n}(\mathcal{J}(T,T^\circ);(\mathcal M_{n-1}^{\Delta})^{\FF_2}):$
$$C = \sum_{
\tiny{\begin{array}{c}
\rho\in\text{supp}(D) \\ \dim(\delta)=1 \text{ and } \delta \subset \min(\rho)^\vee
\end{array}}} (\rho_\infty,\delta)\otimes \iota_{v\wedge v_\delta}\Omega_N.
$$
This concludes the proof. 
\end{proof}

We now investigate what happens when $\delta^{[1]}S = 0$, (i.e. if $[D_{|X_{T,T^{\circ}}}] = 0$). In this case the term $E^{2}_{0,n}$ is equal to $H_n(X_{T^\circ,T}; \mathcal{F}^{\FF_2}_0)$, and thus generated by $S$. More generally the term $E^{k}_{0,n}$ is equal to $H_n(X_{T^\circ,T}; \mathcal{F}^{\FF_2}_0)$ as long as $\delta^{[j]} S = 0$ for all $j < k$.  The next theorem proves the other direction in Theorem \ref{connectedness}.
\begin{thm} If $\delta^{[1]} S = 0$, then
\[ E^{\infty}_{0,n} = H_n(X_{T^\circ,T}; \mathcal{F}^{\FF_2}_0) = \FF_2 \]
 In particular, assuming (\ref{vanishinghypothesis}), if $[D_{|X_{T,T^{\circ}}}] = 0$ then the patchworking $X_{T^\circ,T,\varepsilon}$ has two connected components. 
\end{thm}

\begin{proof} In the proof of Theorem \ref{thm:connectedness} we constructed a lift of $S$ as a chain $\tilde S \in C_n(\mathcal{J}(T^\circ,T);\mathcal{S}^{D})$. By taking its boundary we obtained a (closed) chain $\partial \tilde S \in C_{n-1}(\mathcal{J}(T^\circ,T);\mathcal{K}^{D}_1)$, defining a class $[\partial \tilde S] \in H_{n-1}(\mathcal{J}(T^\circ,T); \mathcal{K}^{D}_1)$. The fact that $\delta^{[1]}S = 0$ means that $[\partial \tilde S]$ maps to zero in $H_{n-1}(\mathcal{J}(T^\circ,T); \mathcal{F}^{\FF_2}_1)$, in particular it must be the image of a class $\alpha_S \in H_{n-1}(\mathcal{J}(T^\circ,T); \mathcal{K}^{D}_2)$. By definition, $\delta^{[2]} S$ is the image of $\alpha_S$ in $H_{n-1}(\mathcal{J}(T^\circ,T); \mathcal{F}^{\FF_2}_2)$. We will prove that 
\[ [\partial \tilde S] = 0 \ \ \text{in} \ \ H_{n-1}(\mathcal{J}(T^\circ,T);\mathcal{K}^{D}_1). \]
This will imply $\delta^{[k]} S = 0$ for all $k \geq 1$ and hence the statement of the theorem. We already proved that as a class in $H_{n-1}(\mathcal{J}(T^\circ,T);\mathcal{M}_1^{\FF_2})$, $\delta^{[1]} S $ is represented by the cycle

\[
\gamma_1=\sum_{\rho\in\text{supp}(D)}\sum_{
\tiny{\begin{array}{c}
(\delta,\rho)\in\mathcal{J}(T^\circ,T) \\ \dim(\delta)=1
\end{array}}} (\delta, \rho)\otimes (\tau_1 - \tau_2),
\]
where $\tau_1$ and $\tau_2$ are the vertices of $\delta$. Since $\delta^{[1]}S =0 $ this cycle must be exact. So there exists a chain 
\[ \zeta = \sum_{\rho \in \Sigma_{T}(1)} \sum_{
\tiny{\begin{array}{c}
(\tau,\rho)\in\mathcal{J}(T^\circ,T) \\ \dim(\tau)=0
\end{array}}}  (\tau, \rho) \otimes z_{(\tau, \rho)}\]
in $C_{n}(\mathcal{J}(T^\circ,T);\mathcal{M}_1^{\FF_2})$ such that 
\[ \partial \zeta = \gamma_1 \]
Notice that the coefficient $z_{(\tau, \rho)} \in \mathcal M_1(\tau, \rho) = \rho^{\perp}$. We then have that the $z_{(\tau, \rho)}$'s satisfy the following conditions 
\begin{itemize}
\item $z_{(\tau, \rho_1)}- z_{(\tau, \rho_2)} = 0 \mod \tau$, when $\rho_1$ and $\rho_2$ are edges of a two dimensional simplex $\lambda \in T$;
\item $z_{(\tau_1, \rho)}- z_{(\tau_2, \rho)} = \tau_1 - \tau_2$, when $\rho \in \supp D$ and $\tau_1, \tau_2$ are vertices of an edge $\delta \in \partial T^{\circ}$ such that $(\delta, \rho) \in \mathcal{J}(T^\circ,T)$;
\item $z_{(\tau_1, \rho)}- z_{(\tau_2, \rho)} = 0$, when $\rho \notin \supp D$ and $\tau_1, \tau_2$ are vertices of an edge $\delta \in \partial T^{\circ}$ such that $(\delta, \rho) \in \mathcal{J}(T^\circ,T)$.
\end{itemize}

For any $\rho \in \Sigma_T(1)$ and some zero dimensional $\tau \in T^{\circ}$ such that $(\tau, \rho) \in \mathcal{J}(T^\circ,T)$, define the following vector
\[ 
v_\rho = \begin{cases}
                       z_{(\tau, \rho)} + \tau \quad \text{if} \ \rho \in \supp D, \\
                       z_{(\tau, \rho)} \quad \text{if} \ \rho \notin \supp D. \\
                  \end{cases}
\]
The second and third bullet above imply that $v_\rho$ is independent of $\tau$, moreover since $z_{(\tau, \rho)} \in \rho^{\perp}$ and $\tau \notin \rho^{\perp}$, we have that $v_\rho \in \mathcal E^D_{(\tau, \rho)}$. We use this vector to define a new lift of $S$ in $C_n(\mathcal{J}(T^\circ,T);\mathcal{S}^{D})$
\[ \tilde S = \sum_{\rho \in \Sigma_T(1)}\sum_{
\tiny{\begin{array}{c}
(\tau,\rho )\in\mathcal{J}(T^\circ,T) \\ \dim(\tau)=0
\end{array}}} (\tau,\rho)\otimes w_{v_\rho}.
 \]
 Then the boundary of $\tilde S$, as a chain in $C_{n-1}(\mathcal{J}(T^\circ,T);\mathcal{K}^{D}_1)$ is given by
\begin{equation} \label{boundary_lift}
   \partial \tilde S =\sum_{
   \tiny{\begin{array}{c}
   (\tau, \lambda)\in\mathcal{J}^S(T^\circ,T) \\ \dim(\tau)=0\text{ and } \dim(\lambda)=2
\end{array}}} (\tau,\lambda)\otimes (w_{v_{\rho_1}}+w_{v_{\rho_2}} ),
\end{equation}
where $\rho_1$ and $\rho_2$ are the two generators of the cone $C(\lambda)$. Notice that there are no contributions from elements of type $(\delta, \rho)$, with $\dim \delta = 1$.  

Suppose that $\lambda \in T$ is a two dimensional simplex such that there are two distinct zero dimensional elements $\tau_1$ and $\tau_2$ such that $(\tau_1, \lambda)$ and $(\tau_2, \lambda)$ are both in $\mathcal{J}(T^\circ,T)$. It follows from the definition of $v_\rho$ that
\[ v_{\rho_j} = z_{(\tau_k, \rho_j)} \mod \tau_k \quad \text{for} \ j,k \in \{1,2 \} \]
and from the first bullet above it follows that 
\begin{equation} \label{v_modt}
	   v_{\rho_1} = v_{\rho_2} \mod \tau_k \quad \text{for} \ k \in \{1,2 \}. 
\end{equation}
Since $T^{\circ}$ is unimodular, we can assume that $\tau_1$ and $\tau_2$ are linearly independent over $\FF_2$, therefore we must have
\[ v_{\rho_1} = v_{\rho_2}. \]
In particular, in \eqref{boundary_lift}, only those $\lambda$'s for which there exists a unique zero dimensional $\tau$ such that $(\tau, \lambda) \in \mathcal{J}(T^\circ,T)$ can contribute with a non-zero coefficient. Observe also that the latter property holds if and only if $(\min \lambda)^\vee$ is zero dimensional and $\tau = (\min \lambda)^\vee$ is a vertex of $\Delta^{\circ}$. In this case $\lambda_{\infty}$ is an edge in the interior of $\tau^\vee$. Therefore we may rewrite \eqref{boundary_lift} in the following way 
\begin{equation*}
	\partial \tilde S = \sum_{\tau \in \operatorname{Vert}(\Delta^\circ)} \sum_{
		\tiny{\begin{array}{c}
				\lambda \in T, \dim(\lambda)=2 \\ \lambda_\infty \subset \inter{\tau^\vee}
	\end{array}}} (\tau,\lambda)\otimes (w_{v_{\rho_1}}+w_{v_{\rho_2}} ).
\end{equation*}

The next goal is to prove that $\partial \tilde S$ is an exact chain in $C_{n-1}(\mathcal{J}(T^\circ,T);\mathcal{K}^{D}_1)$. Indeed, define
\[ \alpha = \sum_{\tau \in \operatorname{Vert}(\Delta^\circ)} \sum_{
	\tiny{\begin{array}{c}
			\lambda \in T, \dim(\lambda)=2 \\ \lambda_\infty \subset \inter{\tau^\vee}
\end{array}}} (\tau,\lambda_{\infty})\otimes (w_{v_{\rho_1}}+w_{v_{\rho_2}} ). \]
Notice that $\alpha$ is supported on the infinite $n$-dimensional cells $F(\tau, \lambda_{\infty})$ which subdivide the cells $F(0, \lambda_{\infty}) \subset X_{T^\circ,T}$. We want to prove the following
\begin{itemize}
	\item $\alpha$ is a well defined chain in $C_{n}(\mathcal{J}(T^\circ,T);\mathcal{K}^{D}_1)$;
	\item $\partial \alpha = \partial \tilde S$.
\end{itemize}
The first bullet amounts to proving that $v_{\rho_1}$ and $v_{\rho_2}$ are both in $\mathcal E^D_{(\tau, \lambda_{\infty})}$, so that the coefficients of $\alpha$ are in $\mathcal K^D_1$. Recall that \eqref{v_modt} holds (with $\tau = \tau_k$), so either $v_{\rho_1}= v_{\rho_2}$, in which case the coefficient of $(\tau, \lambda_\infty)$ is zero, or $v_{\rho_1}= v_{\rho_2} + \tau$. Assume the latter holds. 
Since $\lambda_{\infty} \subset \tau^\vee$, we have $\tau \in \lambda_\infty^\perp$, hence either both  $v_{\rho_1}$ and $v_{\rho_2}$ are in $\mathcal E^D_{(\tau, \lambda_{\infty})}$ or both are not. The cells $(\tau, \lambda_\infty)$, $(\tau, \rho_1)$, $(\tau, \rho_2)$ are the three $n$ dimensional cells intersecting along the $(n-1)$-dimensional cell $(\tau, \lambda)$. We already know that $v_{\rho_1} \in \mathcal E^D_{(\tau, \rho_1)}$, therefore, using the properties of a real phase structure, it is enough to prove that $v_{\rho_1} \notin \mathcal E^D_{(\tau, \rho_2)}$ to conclude that $v_{\rho_1} \in \mathcal E^D_{(\tau, \lambda_{\infty})}$, and hence also $v_{\rho_2}$. We know that $v_{\rho_2} \in \mathcal E^D_{(\tau, \rho_2)}$ and that $\tau \notin \rho_2^\perp$, hence we must have that $v_{\rho_1}= v_{\rho_2} + \tau$ cannot be in $\mathcal E^D_{(\tau, \rho_2)}$. This proves the first of the two bullets. 

Let us now compute $\partial \alpha$. We have 
\[ \begin{split}
	 \partial \alpha  = \sum_{\tau \in \operatorname{Vert}(\Delta^\circ)}    & \sum_{
	        \tiny{\begin{array}{c}
			           \lambda \in T, \dim(\lambda)=2 \\ \lambda_\infty \subset \inter{\tau^\vee}
                  \end{array}}} (\tau,\lambda )\otimes (w_{v_{\rho_1}}+w_{v_{\rho_2}} ) + \\
               \     & +  \sum_{
               	\tiny{\begin{array}{c}
               			\lambda \in T, \dim(\lambda)=2 \\ \lambda_\infty \subset \inter{\tau^\vee}
               \end{array}}} (\hat \tau,\lambda_{\infty})\otimes (w_{[v_{\rho_1}]}+w_{[v_{\rho_2}]} ) + \\
           \     & + \sum_{
           	\tiny{\begin{array}{c}
           			\sigma \in T, \dim(\sigma)=3 \\ \sigma_\infty \subset \inter \tau^\vee
           \end{array}}} (\tau,\sigma_{\infty})\otimes \beta_{(\tau, \sigma)},
  \end{split}
\]
where $[v_{\rho_j}]$ is the image of $v_{\rho_j}$ under the projection $\lambda_{\infty}^{\perp} \rightarrow \frac{\lambda_{\infty}^{\perp}}{\langle \tau \rangle}$. The coefficient $\beta_{(\tau, \sigma)}$ in the above formula is defined as follows. The three dimensional simplex $\sigma$ has three boundary facets $\lambda_1, \lambda_2$ and $\lambda_3$ containing the origin, therefore 
\[ \beta_{(\tau, \sigma)} = \alpha_{(\tau,\lambda_{1, \infty})} +  \alpha_{(\tau,\lambda_{2, \infty})}  +  \alpha_{(\tau,\lambda_{3, \infty})} \]
where $\alpha_{(\tau,\lambda_{j, \infty})}$ denotes the coefficient of $\alpha$ relative to the element $(\tau, \lambda_{j, \infty})$. 

Since $v_{\rho_1} = v_{\rho_2} \mod \tau$, we have that 
\[ w_{[v_{\rho_1}]}+w_{[v_{\rho_2}]} = 0.  \]
Let us now prove that $ \beta_{(\tau, \sigma)} = 0$. Label the edges and vertices of $\sigma_{\infty}$ as in Figure \ref{triangle}. If all three edges are in $\inter \tau^\vee$ we have
\[ \beta_{(\tau, \sigma)} = w_{v_{\rho_2}}+w_{v_{\rho_3}} +  w_{v_{\rho_3}}+w_{v_{\rho_1}} + w_{v_{\rho_1}}+w_{v_{\rho_2}} = 0. \]

\begin{figure}[!ht] 
\begin{center}
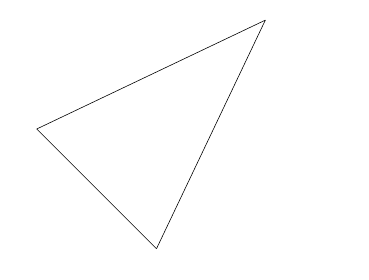
\caption{The boundary of $\lambda_{\infty}$ }
\label{triangle}
\end{center}
\end{figure}

If only one of the edges, say $\lambda_{1, \infty}$, is in the boundary of $\tau^{\vee}$, it means that there exists another $\tau'$ such that $(\tau', \lambda_1) \in \mathcal{J}(T^\circ,T)$, but then $\alpha_{(\tau, \lambda_{1,\infty})} =0$ and
\[ v_{\rho_2} = v_{\rho_3} \] therefore 
\[ \beta_{(\tau, \sigma)} =  w_{v_{\rho_3}}+w_{v_{\rho_1}} + w_{v_{\rho_1}}+w_{v_{\rho_2}} = 0. \]
Similarly,  $\beta_{(\tau, \sigma)} =0$ when two of the edges are in the boundary of $\tau^{\vee}$. This proves that $\partial \alpha = \partial \tilde S$ and concludes the proof of the theorem. \end{proof}

\begin{ex} Let $D = \sum_{a_\rho} D_\rho$ be supported in the interior of the facets of $\Delta$, then $D_{|X_{T,T^{\circ}}} = 0$, since divisors in the interior of maximal faces are mapped to vertices by the blowdown map $\operatorname{Bl}: \T\Sigma_{T} \rightarrow \T\Sigma_{\Delta^{\circ}}$. In particular $X_{T^\circ, T, \epsilon}$ will have two connected components. 
\end{ex}

\begin{ex} Consider the divisors $D= D_7 + D_8$ and $D= D_8$ as in Figure \ref{pw_divisors}. In the first case $D_{|X_{T,T^{\circ}}}$ is one point, in the second $D_{|X_{T,T^{\circ}}} = 0$ (see Figure \ref{div_conn}). The corresponding patchworkings are depicted in Figure \ref{pw_divisors} and they are respectivley connected and disconnected. 

\begin{figure}[!ht] 
\begin{center}
\includegraphics{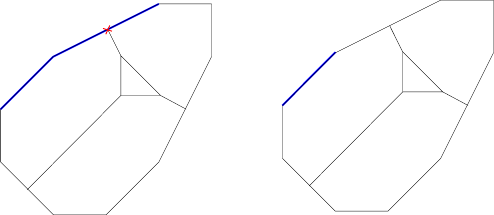}
\caption{The intersection of $X_{T, T^\circ}$ with two different divisors.}
\label{div_conn}
\end{center}
\end{figure}
\end{ex}

\begin{ex} Consider the two reflexive polytopes in Figure \ref{pw_divisors_2}, where $T$ is on the left. We have drawn an example of patchworking and the corresponding divisor. The divisor $D_{|X_{T,T^{\circ}}}$ consists of two points and therefore it is zero ($\mod 2$). Indeed any patchworking in this example will be disconnected. 
\begin{figure}[!ht] 
\begin{center}
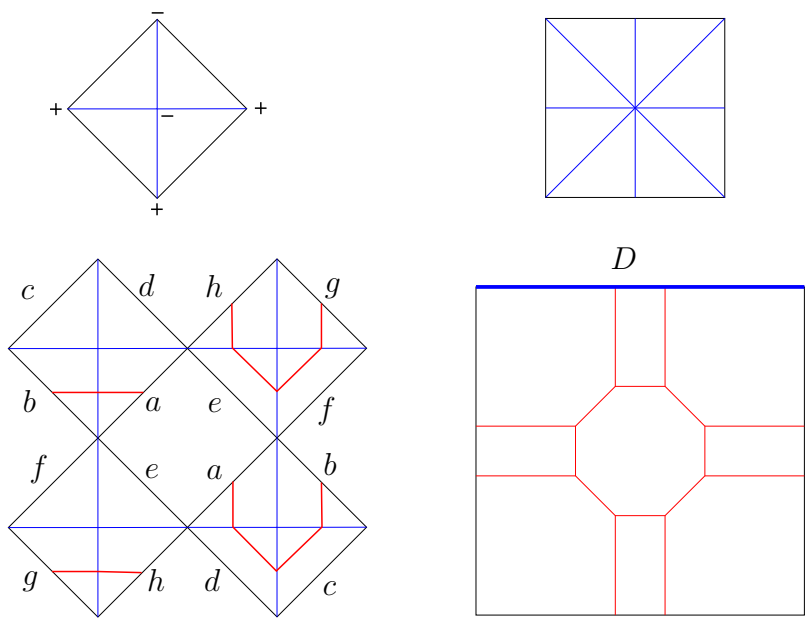
\caption{Two dual reflexive polygons. A divisor $D$ in the mirror and the corresponding disconnected patchworking. The divisor intersects the mirror curve in two points, i.e. $D_{|X_{T,T^{\circ}}}=0$. }
\label{pw_divisors_2}
\end{center}
\end{figure}
\end{ex}

\vspace{1cm}
\begin{flushleft}
Diego MATESSI \\
Dipartimento di Matematica \\
Universit\`a degli Studi di Milano \\
Via Saldini 50 \\
I-20133 Milano, Italy \\
E-mail address: \email{diego.matessi@unimi.it}
\end{flushleft}
\vspace{1cm}
\begin{flushleft}
Arthur RENAUDINEAU \\
Univ. Lille, CNRS, UMR 8524 - Laboratoire Paul Painlev\'e \\
F-59000 Lille, France\\
E-mail address: \email{arthur.renaudineau@univ-lille.fr}
\end{flushleft}

\end{document}

%% file: refl_pol_subdiv_2.pdf_tex
\begingroup%
  \makeatletter%
  \providecommand\color[2][]{%
    \errmessage{(Inkscape) Color is used for the text in Inkscape, but the package 'color.sty' is not loaded}%
    \renewcommand\color[2][]{}%
  }%
  \providecommand\transparent[1]{%
    \errmessage{(Inkscape) Transparency is used (non-zero) for the text in Inkscape, but the package 'transparent.sty' is not loaded}%
    \renewcommand\transparent[1]{}%
  }%
  \providecommand\rotatebox[2]{#2}%
  \newcommand*\fsize{\dimexpr\f@size pt\relax}%
  \newcommand*\lineheight[1]{\fontsize{\fsize}{#1\fsize}\selectfont}%
  \ifx\svgwidth\undefined%
    \setlength{\unitlength}{328.18802721bp}%
    \ifx\svgscale\undefined%
      \relax%
    \else%
      \setlength{\unitlength}{\unitlength * \real{\svgscale}}%
    \fi%
  \else%
    \setlength{\unitlength}{\svgwidth}%
  \fi%
  \global\let\svgwidth\undefined%
  \global\let\svgscale\undefined%
  \makeatother%
  \begin{picture}(1,0.62157242)%
    \lineheight{1}%
    \setlength\tabcolsep{0pt}%
    \put(0,0){\includegraphics[width=\unitlength,page=1]{refl_pol_subdiv_2.pdf}}%
    \put(0.81627307,0.40968117){\color[rgb]{0,0,0}\makebox(0,0)[lt]{\lineheight{1.25}\smash{\begin{tabular}[t]{l}$\sigma_1$\end{tabular}}}}%
    \put(0.68604638,0.6075037){\color[rgb]{0,0,0}\makebox(0,0)[lt]{\lineheight{1.25}\smash{\begin{tabular}[t]{l}$\sigma_2$\end{tabular}}}}%
    \put(0.63956865,0.42264916){\color[rgb]{0,0,0}\makebox(0,0)[lt]{\lineheight{1.25}\smash{\begin{tabular}[t]{l}$\sigma_3$\end{tabular}}}}%
    \put(0,0){\includegraphics[width=\unitlength,page=2]{refl_pol_subdiv_2.pdf}}%
    \put(0.26710927,0.47090511){\color[rgb]{0,0,0}\makebox(0,0)[lt]{\lineheight{1.25}\smash{\begin{tabular}[t]{l}$(\tau_2,\sigma_2)$\end{tabular}}}}%
    \put(0,0){\includegraphics[width=\unitlength,page=3]{refl_pol_subdiv_2.pdf}}%
    \put(0.83166843,0.09930704){\color[rgb]{0,0,0}\makebox(0,0)[lt]{\lineheight{1.25}\smash{\begin{tabular}[t]{l}$\tau_1$\end{tabular}}}}%
    \put(0,0){\includegraphics[width=\unitlength,page=4]{refl_pol_subdiv_2.pdf}}%
    \put(0.90953677,0.15158652){\color[rgb]{0,0,0}\makebox(0,0)[lt]{\lineheight{1.25}\smash{\begin{tabular}[t]{l}$\tau_2$\end{tabular}}}}%
    \put(0,0){\includegraphics[width=\unitlength,page=5]{refl_pol_subdiv_2.pdf}}%
    \put(-0.00138366,0.17733214){\color[rgb]{0,0,0}\makebox(0,0)[lt]{\lineheight{1.25}\smash{\begin{tabular}[t]{l}$(\tau_3, \sigma_3)$\end{tabular}}}}%
    \put(0,0){\includegraphics[width=\unitlength,page=6]{refl_pol_subdiv_2.pdf}}%
    \put(0.79514234,0.14347761){\color[rgb]{0,0,0}\makebox(0,0)[lt]{\lineheight{1.25}\smash{\begin{tabular}[t]{l}$\tau_3$\end{tabular}}}}%
    \put(0,0){\includegraphics[width=\unitlength,page=7]{refl_pol_subdiv_2.pdf}}%
    \put(0.22820906,0.17506929){\color[rgb]{0,0,0}\makebox(0,0)[lt]{\lineheight{1.25}\smash{\begin{tabular}[t]{l}$(\tau_1, \sigma_1)$\end{tabular}}}}%
  \end{picture}%
\endgroup%

%% file: refl_pol_subdiv.pdf_tex
\begingroup%
  \makeatletter%
  \providecommand\color[2][]{%
    \errmessage{(Inkscape) Color is used for the text in Inkscape, but the package 'color.sty' is not loaded}%
    \renewcommand\color[2][]{}%
  }%
  \providecommand\transparent[1]{%
    \errmessage{(Inkscape) Transparency is used (non-zero) for the text in Inkscape, but the package 'transparent.sty' is not loaded}%
    \renewcommand\transparent[1]{}%
  }%
  \providecommand\rotatebox[2]{#2}%
  \newcommand*\fsize{\dimexpr\f@size pt\relax}%
  \newcommand*\lineheight[1]{\fontsize{\fsize}{#1\fsize}\selectfont}%
  \ifx\svgwidth\undefined%
    \setlength{\unitlength}{369.14975144bp}%
    \ifx\svgscale\undefined%
      \relax%
    \else%
      \setlength{\unitlength}{\unitlength * \real{\svgscale}}%
    \fi%
  \else%
    \setlength{\unitlength}{\svgwidth}%
  \fi%
  \global\let\svgwidth\undefined%
  \global\let\svgscale\undefined%
  \makeatother%
  \begin{picture}(1,0.60594713)%
    \lineheight{1}%
    \setlength\tabcolsep{0pt}%
    \put(0.79762933,0.00335709){\color[rgb]{0,0,0}\makebox(0,0)[lt]{\lineheight{1.25}\smash{\begin{tabular}[t]{l}$\sigma_1$\end{tabular}}}}%
    \put(0.90368496,0.27210718){\color[rgb]{0,0,0}\makebox(0,0)[lt]{\lineheight{1.25}\smash{\begin{tabular}[t]{l}$\sigma_2$\end{tabular}}}}%
    \put(0,0){\includegraphics[width=\unitlength,page=1]{refl_pol_subdiv.pdf}}%
    \put(0.6454569,0.14643886){\color[rgb]{0,0,0}\makebox(0,0)[lt]{\lineheight{1.25}\smash{\begin{tabular}[t]{l}$\sigma_3$\end{tabular}}}}%
    \put(0,0){\includegraphics[width=\unitlength,page=2]{refl_pol_subdiv.pdf}}%
    \put(0.31378735,0.23896116){\color[rgb]{0,0,0}\makebox(0,0)[lt]{\lineheight{1.25}\smash{\begin{tabular}[t]{l}$(\tau_1, \sigma_1)$\end{tabular}}}}%
    \put(0.78418089,0.43386148){\color[rgb]{0,0,0}\makebox(0,0)[lt]{\lineheight{1.25}\smash{\begin{tabular}[t]{l}$\tau_1$\end{tabular}}}}%
    \put(0.67611209,0.5934395){\color[rgb]{0,0,0}\makebox(0,0)[lt]{\lineheight{1.25}\smash{\begin{tabular}[t]{l}$\tau_2$\end{tabular}}}}%
    \put(0.33100966,0.57177526){\color[rgb]{0,0,0}\makebox(0,0)[lt]{\lineheight{1.25}\smash{\begin{tabular}[t]{l}$(\tau_2, \sigma_2)$\end{tabular}}}}%
    \put(0,0){\includegraphics[width=\unitlength,page=3]{refl_pol_subdiv.pdf}}%
    \put(-0.00123013,0.33980758){\color[rgb]{0,0,0}\makebox(0,0)[lt]{\lineheight{1.25}\smash{\begin{tabular}[t]{l}$(\tau_3, \sigma_3)$\end{tabular}}}}%
    \put(0,0){\includegraphics[width=\unitlength,page=4]{refl_pol_subdiv.pdf}}%
    \put(0.69745104,0.43821992){\color[rgb]{0,0,0}\makebox(0,0)[lt]{\lineheight{1.25}\smash{\begin{tabular}[t]{l}$\tau_3$\end{tabular}}}}%
  \end{picture}%
\endgroup%

%% file: refined_subdiv.pdf_tex
\begingroup%
  \makeatletter%
  \providecommand\color[2][]{%
    \errmessage{(Inkscape) Color is used for the text in Inkscape, but the package 'color.sty' is not loaded}%
    \renewcommand\color[2][]{}%
  }%
  \providecommand\transparent[1]{%
    \errmessage{(Inkscape) Transparency is used (non-zero) for the text in Inkscape, but the package 'transparent.sty' is not loaded}%
    \renewcommand\transparent[1]{}%
  }%
  \providecommand\rotatebox[2]{#2}%
  \newcommand*\fsize{\dimexpr\f@size pt\relax}%
  \newcommand*\lineheight[1]{\fontsize{\fsize}{#1\fsize}\selectfont}%
  \ifx\svgwidth\undefined%
    \setlength{\unitlength}{342.43180162bp}%
    \ifx\svgscale\undefined%
      \relax%
    \else%
      \setlength{\unitlength}{\unitlength * \real{\svgscale}}%
    \fi%
  \else%
    \setlength{\unitlength}{\svgwidth}%
  \fi%
  \global\let\svgwidth\undefined%
  \global\let\svgscale\undefined%
  \makeatother%
  \begin{picture}(1,0.56798048)%
    \lineheight{1}%
    \setlength\tabcolsep{0pt}%
    \put(0,0){\includegraphics[width=\unitlength,page=1]{refined_subdiv.pdf}}%
    \put(0.74411807,0.14704212){\color[rgb]{0,0,0}\makebox(0,0)[lt]{\lineheight{1.25}\smash{\begin{tabular}[t]{l}$\sigma$\end{tabular}}}}%
    \put(0,0){\includegraphics[width=\unitlength,page=2]{refined_subdiv.pdf}}%
    \put(0.77612349,0.38015003){\color[rgb]{0,0,0}\makebox(0,0)[lt]{\lineheight{1.25}\smash{\begin{tabular}[t]{l}$\tau$\end{tabular}}}}%
    \put(0,0){\includegraphics[width=\unitlength,page=3]{refined_subdiv.pdf}}%
    \put(0.13790912,0.26473134){\color[rgb]{0,0,0}\makebox(0,0)[lt]{\lineheight{1.25}\smash{\begin{tabular}[t]{l}$(\tau_\infty, \sigma)$\end{tabular}}}}%
    \put(0.38866455,0.34578212){\color[rgb]{0,0,0}\makebox(0,0)[lt]{\lineheight{1.25}\smash{\begin{tabular}[t]{l}$(\tau_\infty, \hat \sigma)$\end{tabular}}}}%
    \put(-0.00024801,0.21288246){\color[rgb]{0,0,0}\makebox(0,0)[lt]{\lineheight{1.25}\smash{\begin{tabular}[t]{l}$(\tau, \sigma)$\end{tabular}}}}%
  \end{picture}%
\endgroup%

%% file: refined_subdiv_2.pdf_tex
\begingroup%
  \makeatletter%
  \providecommand\color[2][]{%
    \errmessage{(Inkscape) Color is used for the text in Inkscape, but the package 'color.sty' is not loaded}%
    \renewcommand\color[2][]{}%
  }%
  \providecommand\transparent[1]{%
    \errmessage{(Inkscape) Transparency is used (non-zero) for the text in Inkscape, but the package 'transparent.sty' is not loaded}%
    \renewcommand\transparent[1]{}%
  }%
  \providecommand\rotatebox[2]{#2}%
  \newcommand*\fsize{\dimexpr\f@size pt\relax}%
  \newcommand*\lineheight[1]{\fontsize{\fsize}{#1\fsize}\selectfont}%
  \ifx\svgwidth\undefined%
    \setlength{\unitlength}{345.22989829bp}%
    \ifx\svgscale\undefined%
      \relax%
    \else%
      \setlength{\unitlength}{\unitlength * \real{\svgscale}}%
    \fi%
  \else%
    \setlength{\unitlength}{\svgwidth}%
  \fi%
  \global\let\svgwidth\undefined%
  \global\let\svgscale\undefined%
  \makeatother%
  \begin{picture}(1,0.88889448)%
    \lineheight{1}%
    \setlength\tabcolsep{0pt}%
    \put(0,0){\includegraphics[width=\unitlength,page=1]{refined_subdiv_2.pdf}}%
    \put(0.10117098,0.2109611){\color[rgb]{0,0,0}\makebox(0,0)[lt]{\lineheight{1.25}\smash{\begin{tabular}[t]{l}$(\sigma_\infty, \tau)$\end{tabular}}}}%
    \put(0,0){\includegraphics[width=\unitlength,page=2]{refined_subdiv_2.pdf}}%
    \put(0.72005906,0.49664308){\color[rgb]{0,0,0}\makebox(0,0)[lt]{\lineheight{1.25}\smash{\begin{tabular}[t]{l}$\sigma$\end{tabular}}}}%
    \put(0,0){\includegraphics[width=\unitlength,page=3]{refined_subdiv_2.pdf}}%
    \put(-0.00131536,0.20833655){\color[rgb]{0,0,0}\makebox(0,0)[lt]{\lineheight{1.25}\smash{\begin{tabular}[t]{l}$(\sigma, \tau)$\end{tabular}}}}%
    \put(0.22822166,0.21127855){\color[rgb]{0,0,0}\makebox(0,0)[lt]{\lineheight{1.25}\smash{\begin{tabular}[t]{l}$(\sigma_\infty, \hat \tau)$\end{tabular}}}}%
    \put(0.61453186,0.72171648){\color[rgb]{0,0,0}\makebox(0,0)[lt]{\lineheight{1.25}\smash{\begin{tabular}[t]{l}$\tau$\end{tabular}}}}%
    \put(0,0){\includegraphics[width=\unitlength,page=4]{refined_subdiv_2.pdf}}%
  \end{picture}%
\endgroup%

%% file: pw_divisors.pdf_tex
\begingroup%
  \makeatletter%
  \providecommand\color[2][]{%
    \errmessage{(Inkscape) Color is used for the text in Inkscape, but the package 'color.sty' is not loaded}%
    \renewcommand\color[2][]{}%
  }%
  \providecommand\transparent[1]{%
    \errmessage{(Inkscape) Transparency is used (non-zero) for the text in Inkscape, but the package 'transparent.sty' is not loaded}%
    \renewcommand\transparent[1]{}%
  }%
  \providecommand\rotatebox[2]{#2}%
  \newcommand*\fsize{\dimexpr\f@size pt\relax}%
  \newcommand*\lineheight[1]{\fontsize{\fsize}{#1\fsize}\selectfont}%
  \ifx\svgwidth\undefined%
    \setlength{\unitlength}{290.26672604bp}%
    \ifx\svgscale\undefined%
      \relax%
    \else%
      \setlength{\unitlength}{\unitlength * \real{\svgscale}}%
    \fi%
  \else%
    \setlength{\unitlength}{\svgwidth}%
  \fi%
  \global\let\svgwidth\undefined%
  \global\let\svgscale\undefined%
  \makeatother%
  \begin{picture}(1,0.71131055)%
    \lineheight{1}%
    \setlength\tabcolsep{0pt}%
    \put(0,0){\includegraphics[width=\unitlength,page=1]{pw_divisors.pdf}}%
    \put(-0.00082268,0.41416426){\color[rgb]{0,0,0}\makebox(0,0)[lt]{\lineheight{1.25}\smash{\begin{tabular}[t]{l}$D_1$\end{tabular}}}}%
    \put(0.11321279,0.41345675){\color[rgb]{0,0,0}\makebox(0,0)[lt]{\lineheight{1.25}\smash{\begin{tabular}[t]{l}$D_2$\end{tabular}}}}%
    \put(0.18988135,0.41360461){\color[rgb]{0,0,0}\makebox(0,0)[lt]{\lineheight{1.25}\smash{\begin{tabular}[t]{l}$D_3$\end{tabular}}}}%
    \put(0.27543841,0.41368076){\color[rgb]{0,0,0}\makebox(0,0)[lt]{\lineheight{1.25}\smash{\begin{tabular}[t]{l}$D_4$\end{tabular}}}}%
    \put(0.2272141,0.52618707){\color[rgb]{0,0,0}\makebox(0,0)[lt]{\lineheight{1.25}\smash{\begin{tabular}[t]{l}$D_5$\end{tabular}}}}%
    \put(0.14014621,0.6125866){\color[rgb]{0,0,0}\makebox(0,0)[lt]{\lineheight{1.25}\smash{\begin{tabular}[t]{l}$D_6$\end{tabular}}}}%
    \put(0.01500848,0.69540384){\color[rgb]{0,0,0}\makebox(0,0)[lt]{\lineheight{1.25}\smash{\begin{tabular}[t]{l}$D_7$\end{tabular}}}}%
    \put(-0.00156443,0.58316697){\color[rgb]{0,0,0}\makebox(0,0)[lt]{\lineheight{1.25}\smash{\begin{tabular}[t]{l}$D_8$\end{tabular}}}}%
    \put(-0.00145438,0.5218016){\color[rgb]{0,0,0}\makebox(0,0)[lt]{\lineheight{1.25}\smash{\begin{tabular}[t]{l}$D_9$\end{tabular}}}}%
    \put(0,0){\includegraphics[width=\unitlength,page=2]{pw_divisors.pdf}}%
  \end{picture}%
\endgroup%

%% file: triangle.pdf_tex
\begingroup%
  \makeatletter%
  \providecommand\color[2][]{%
    \errmessage{(Inkscape) Color is used for the text in Inkscape, but the package 'color.sty' is not loaded}%
    \renewcommand\color[2][]{}%
  }%
  \providecommand\transparent[1]{%
    \errmessage{(Inkscape) Transparency is used (non-zero) for the text in Inkscape, but the package 'transparent.sty' is not loaded}%
    \renewcommand\transparent[1]{}%
  }%
  \providecommand\rotatebox[2]{#2}%
  \newcommand*\fsize{\dimexpr\f@size pt\relax}%
  \newcommand*\lineheight[1]{\fontsize{\fsize}{#1\fsize}\selectfont}%
  \ifx\svgwidth\undefined%
    \setlength{\unitlength}{177.06417414bp}%
    \ifx\svgscale\undefined%
      \relax%
    \else%
      \setlength{\unitlength}{\unitlength * \real{\svgscale}}%
    \fi%
  \else%
    \setlength{\unitlength}{\svgwidth}%
  \fi%
  \global\let\svgwidth\undefined%
  \global\let\svgscale\undefined%
  \makeatother%
  \begin{picture}(1,0.75727104)%
    \lineheight{1}%
    \setlength\tabcolsep{0pt}%
    \put(0,0){\includegraphics[width=\unitlength,page=1]{triangle.pdf}}%
    \put(0.2567298,0.57475089){\color[rgb]{0,0,0}\makebox(0,0)[lt]{\lineheight{1.25}\smash{\begin{tabular}[t]{l}$\lambda_{1, \infty}$\end{tabular}}}}%
    \put(0.1412034,0.19180225){\color[rgb]{0,0,0}\makebox(0,0)[lt]{\lineheight{1.25}\smash{\begin{tabular}[t]{l}$\lambda_{2, \infty}$\end{tabular}}}}%
    \put(0.59475155,0.36081311){\color[rgb]{0,0,0}\makebox(0,0)[lt]{\lineheight{1.25}\smash{\begin{tabular}[t]{l}$\lambda_{3, \infty}$\end{tabular}}}}%
    \put(0.4064862,0.00995511){\color[rgb]{0,0,0}\makebox(0,0)[lt]{\lineheight{1.25}\smash{\begin{tabular}[t]{l}$\rho_1$\end{tabular}}}}%
    \put(0.74022923,0.71381046){\color[rgb]{0,0,0}\makebox(0,0)[lt]{\lineheight{1.25}\smash{\begin{tabular}[t]{l}$\rho_2$\end{tabular}}}}%
    \put(-0.00427436,0.39076439){\color[rgb]{0,0,0}\makebox(0,0)[lt]{\lineheight{1.25}\smash{\begin{tabular}[t]{l}$\rho_3$\end{tabular}}}}%
  \end{picture}%
\endgroup%

%% file: pw_divisors_2.pdf_tex
\begingroup%
  \makeatletter%
  \providecommand\color[2][]{%
    \errmessage{(Inkscape) Color is used for the text in Inkscape, but the package 'color.sty' is not loaded}%
    \renewcommand\color[2][]{}%
  }%
  \providecommand\transparent[1]{%
    \errmessage{(Inkscape) Transparency is used (non-zero) for the text in Inkscape, but the package 'transparent.sty' is not loaded}%
    \renewcommand\transparent[1]{}%
  }%
  \providecommand\rotatebox[2]{#2}%
  \newcommand*\fsize{\dimexpr\f@size pt\relax}%
  \newcommand*\lineheight[1]{\fontsize{\fsize}{#1\fsize}\selectfont}%
  \ifx\svgwidth\undefined%
    \setlength{\unitlength}{302.76129535bp}%
    \ifx\svgscale\undefined%
      \relax%
    \else%
      \setlength{\unitlength}{\unitlength * \real{\svgscale}}%
    \fi%
  \else%
    \setlength{\unitlength}{\svgwidth}%
  \fi%
  \global\let\svgwidth\undefined%
  \global\let\svgscale\undefined%
  \makeatother%
  \begin{picture}(1,0.75414243)%
    \lineheight{1}%
    \setlength\tabcolsep{0pt}%
    \put(0,0){\includegraphics[width=\unitlength,page=1]{pw_divisors_2.pdf}}%
    \put(0.25543725,0.17266883){\color[rgb]{0,0,0}\makebox(0,0)[lt]{\lineheight{1.25}\smash{\begin{tabular}[t]{l}$a$\end{tabular}}}}%
    \put(0.17903756,0.25794343){\color[rgb]{0,0,0}\makebox(0,0)[lt]{\lineheight{1.25}\smash{\begin{tabular}[t]{l}$a$\end{tabular}}}}%
    \put(0.40119973,0.17710635){\color[rgb]{0,0,0}\makebox(0,0)[lt]{\lineheight{1.25}\smash{\begin{tabular}[t]{l}$b$\end{tabular}}}}%
    \put(0.02678491,0.25645759){\color[rgb]{0,0,0}\makebox(0,0)[lt]{\lineheight{1.25}\smash{\begin{tabular}[t]{l}$b$\end{tabular}}}}%
    \put(0.40039792,0.02872442){\color[rgb]{0,0,0}\makebox(0,0)[lt]{\lineheight{1.25}\smash{\begin{tabular}[t]{l}$c$\end{tabular}}}}%
    \put(0.02416541,0.39719555){\color[rgb]{0,0,0}\makebox(0,0)[lt]{\lineheight{1.25}\smash{\begin{tabular}[t]{l}$c$\end{tabular}}}}%
    \put(0.25268001,0.03363139){\color[rgb]{0,0,0}\makebox(0,0)[lt]{\lineheight{1.25}\smash{\begin{tabular}[t]{l}$d$\end{tabular}}}}%
    \put(0.17070924,0.40028471){\color[rgb]{0,0,0}\makebox(0,0)[lt]{\lineheight{1.25}\smash{\begin{tabular}[t]{l}$d$\end{tabular}}}}%
    \put(0.25713753,0.25657502){\color[rgb]{0,0,0}\makebox(0,0)[lt]{\lineheight{1.25}\smash{\begin{tabular}[t]{l}$e$\end{tabular}}}}%
    \put(0.17847054,0.17392002){\color[rgb]{0,0,0}\makebox(0,0)[lt]{\lineheight{1.25}\smash{\begin{tabular}[t]{l}$e$\end{tabular}}}}%
    \put(0.39298856,0.24689844){\color[rgb]{0,0,0}\makebox(0,0)[lt]{\lineheight{1.25}\smash{\begin{tabular}[t]{l}$f$\end{tabular}}}}%
    \put(0.03497543,0.17526828){\color[rgb]{0,0,0}\makebox(0,0)[lt]{\lineheight{1.25}\smash{\begin{tabular}[t]{l}$f$\end{tabular}}}}%
    \put(0.40266522,0.4055238){\color[rgb]{0,0,0}\makebox(0,0)[lt]{\lineheight{1.25}\smash{\begin{tabular}[t]{l}$g$\end{tabular}}}}%
    \put(0.25395127,0.39995266){\color[rgb]{0,0,0}\makebox(0,0)[lt]{\lineheight{1.25}\smash{\begin{tabular}[t]{l}$h$\end{tabular}}}}%
    \put(0.18175424,0.03294722){\color[rgb]{0,0,0}\makebox(0,0)[lt]{\lineheight{1.25}\smash{\begin{tabular}[t]{l}$h$\end{tabular}}}}%
    \put(0.0265299,0.04012171){\color[rgb]{0,0,0}\makebox(0,0)[lt]{\lineheight{1.25}\smash{\begin{tabular}[t]{l}$g$\end{tabular}}}}%
    \put(0,0){\includegraphics[width=\unitlength,page=2]{pw_divisors_2.pdf}}%
    \put(0.7582124,0.43371527){\color[rgb]{0,0,0}\makebox(0,0)[lt]{\lineheight{1.25}\smash{\begin{tabular}[t]{l}$D$\end{tabular}}}}%
  \end{picture}%
\endgroup%